\documentclass[]{article}

\usepackage{graphicx}
\usepackage{amsmath}
\usepackage{amssymb}
\usepackage{amsthm}
\usepackage{pxfonts}
\usepackage{enumerate}
\usepackage{color}
\usepackage{mathdots}
\usepackage{sectsty}
\usepackage[hidelinks]{hyperref}
\usepackage{tikz}
\usepackage{caption}
\usepackage{adjustbox}

\sectionfont{\scshape\centering\fontsize{11}{14}\selectfont}
\subsectionfont{\scshape\fontsize{11}{14}\selectfont}
\usepackage{fancyhdr}

\newcommand\shorttitle{On the joint moments of the characteristic polynomials of random unitary matrices}
\newcommand\authors{Theodoros Assiotis, Jonathan P. Keating and Jon Warren}

\newtheorem{thm}{Theorem}[section]
\newtheorem{cor}[thm]{Corollary}
\newtheorem{lem}[thm]{Lemma}
\newtheorem{defn}[thm]{Definition}
\newtheorem{rmk}[thm]{Remark}
\newtheorem{prop}[thm]{Proposition}
\newtheorem{conj}[thm]{Conjecture}

\title{\large \bf ON THE JOINT MOMENTS OF THE CHARACTERISTIC POLYNOMIALS OF RANDOM UNITARY MATRICES}
\author{\small THEODOROS ASSIOTIS, JONATHAN P. KEATING AND JON WARREN}
\date{}

\begin{document}

\maketitle

\begin{abstract}
We establish the asymptotics of the joint moments of the characteristic polynomial of a random unitary matrix and its derivative for general real values of the exponents, proving a conjecture made by Hughes in 2001 \cite{Hughes}. Moreover, we give a probabilistic representation for the leading order coefficient in the asymptotic in terms of a real-valued random variable that plays an important role in the ergodic decomposition of the Hua-Pickrell measures. This enables us to establish connections between the characteristic function of this random variable and the $\sigma$-Painlev\'{e} III' equation.
\end{abstract}

\tableofcontents

\section{Introduction}

\subsection{History of the problem}

Let $\mathbf{U} \in \mathbb{U}(N)$ where $\mathbb{U}(N)$ is the group of $N\times N$ unitary matrices. We consider the characteristic polynomial of $\mathbf{U}$ given by:
\begin{align*}
Z_\mathbf{U}(\theta)=\det \left(\mathbf{I}-e^{-\i \theta}\mathbf{U}\right)=\prod_{j=1}^{N}\left(1-e^{\i(\theta_j-\theta)}\right)
\end{align*}
where $e^{\i \theta_1},\dots,e^{\i \theta_N}$ are the eigenvalues of $\mathbf{U}$ with $\theta_i \in [0,2\pi)$.
We define the closely related quantity:
\begin{align*}
V_{\mathbf{U}}(\theta)=e^{\i N\frac{\theta+\pi}{2}-\i \sum_{j=1}^{N}\frac{\theta_j}{2}}Z_\mathbf{U}(\theta).
\end{align*}
Note that, $V_\mathbf{U}(\theta)$ is real valued for $\theta \in [0,2\pi)$ and that $|V_\mathbf{U}(\theta)|=|Z_\mathbf{U}(\theta)|$. We are interested in the following quantities, which we call the joint moments and denote by $\mathfrak{F}_{N}(s,h)$, see \cite{JointMomentsPainleve},\cite{Hughes}, \cite{MixedMoments}, given by:
\begin{align}
\mathfrak{F}_N(s,h)=\int_{\mathbb{U}(N)}^{}|V_\mathbf{U}(0)|^{2s-2h}|V'_\mathbf{U}(0)|^{2h}d\mu_N(\mathbf{U})
\end{align}
where $d\mu_N$ denotes the normalized Haar probability measure on $\mathbb{U}(N)$, with $-\frac{1}{2}<h<s+\frac{1}{2}$. In addition to their intrinsic probabilistic interest, our motivation behind studying these quantities comes from connections between the characteristic polynomials of random unitary matrices and number theory, in particular the Riemann zeta function.  We will expand on this in Section \ref{NumberTheorySection} below.

The goal of this paper is to investigate the $N \to \infty$ behaviour of $\mathfrak{F}_{N}(s,h)$ and, in particular, to prove the following conjecture for its asymptotics made in 2001 by Chris Hughes \cite{Hughes}:
\begin{align}
\frac{1}{N^{s^2+2h}}\mathfrak{F}_N(s,h)\overset{?}{\longrightarrow}\mathfrak{F}(s,h), \ \textnormal{ as } N \to \infty,
\end{align}
for a certain function $\mathfrak{F}(s,h)$, unknown for general real values of the parameters $s$ and $h$.

We start by giving a brief overview of what was known before our work. In his thesis \cite{Hughes} Hughes proved the conjecture for $s,h \in \mathbb{N}$ and gave an expression for $\mathfrak{F}(s,h)$ in terms of a combinatorial sum. This expression for $\mathfrak{F}(s,h)$ made sense for real $s$ as well, but not general real $h$. Subsequently Conrey, Rubinstein and Snaith \cite{ConreyRubinsteinSnaith}, by making use of certain multiple contour integral formulae for these quantities from \cite{Autocorrelation} and \cite{Integralmoments}, gave an alternative proof in the special case $s=h\in \mathbb{N}$ and obtained a different expression for the leading order coefficient $\mathfrak{F}(s,s)$ in the form of a determinant of Bessel functions.  This determinant was later connected to Painlev\'{e} equations by Forrester and Witte \cite{ForresterWitte2}. Dehaye \cite{Dehaye1}, \cite{Dehaye2} gave yet another proof of the conjecture for $s,h\in \mathbb{N}$ using representation theory and symmetric functions and obtained an expression for $\mathfrak{F}(s,h)$ in terms of a different combinatorial sum. Again, it was observed that the expression for $\mathfrak{F}(s,h)$ made sense for real values of $s$ as well, but $h$ had to be an integer. Winn in \cite{Winn}, by finding connections with hypergeometric functions, was the first to prove the conjecture for $s\in \mathbb{N}$ and half-integer $h \in \frac{1}{2}\mathbb{N}$ and gave an expression for $\mathfrak{F}(s,h)$ in terms of a combinatorial sum.  Recently, Basor {\it et al.} \cite{JointMomentsPainleve}  obtained an alternative proof of the conjecture for $s,h \in \mathbb{N}$, by connecting the pre-limit quantity $\mathfrak{F}_N(s,h)$ to Painlev\'{e} equations using Riemann Hilbert problem techniques, and gave an expression for $\mathfrak{F}(s,h)$ in terms of Painlev\'{e} transcendents. They also conjectured conformal block expansions for the joint moments.  Finally, Bailey {\it et al.} \cite{MixedMoments} extended the approach of \cite{ConreyRubinsteinSnaith} to $s,h \in \mathbb{N}$ and obtained analogous results for the $N\to\infty$ limit. 

It is worth emphasizing that even though there was some indication in several of the works just listed of what $\mathfrak{F}(s,h)$ might be for real values of $s$, there was no proof of the asymptotic formula for non-integer values of this variable.  Moreover, nothing was known or conjectured, prior to the present work, for general real values of $h$. In this paper, by developing a novel approach, we prove the conjecture for arbitrary real values of $s>-\frac{1}{2}$ and positive real values of $h$ in the full range $h<s+\frac{1}{2}$ and moreover identify the function $\mathfrak{F}(s,h)$, which turns out to have an interesting probabilistic interpretation.

\subsection{The main result}

In order to state our main result we first require the definition of a certain remarkable real-valued random variable. This random variable might seem like it comes out of thin air, but we shall explain its origins and significance, and also how it arises in the argument, in Sections \ref{SubsectionOrigins} and \ref{SubsectionStrategyProof}.

\begin{defn}\label{DefinitionGamma1}
Let $s>-\frac{1}{2}$. We consider the determinantal point process $\mathsf{C}^{(s)}$ on $\mathbb{R}^*=(-\infty,0)\cup (0,\infty)$ with correlation kernel $\mathsf{K}^{(s)}(x,y)$\footnote{This means that $\mathsf{C}^{(s)}$ is a random point process on $\mathbb{R}^*$ so that for any $l\ge 1$ and $F$ a bounded Borel function of compact support on $\left(\mathbb{R}^*\right)^l$:\begin{align*}
\mathbb{E}\left[\sum_{x_{i_1},\dots,x_{i_l}\in \mathsf{C}^{(s)}}^{}F(x_{i_1},\dots,x_{i_l})\right]=\int_{\left(\mathbb{R}^*\right)^l}^{}F(z_1,\dots,z_l)\det\left[\mathsf{K}^{(s)}\left(z_i,z_j\right)\right]_{i,j=1}^ldz_1\dots dz_l,
\end{align*}
where the sum is over all $l$-tuples of pairwise distinct points of the random point configuration $\mathsf{C}^{(s)}$. Or equivalently for any $l\ge 1$ and distinct points $z_1,\dots,z_l \in \mathbb{R}^*$:
\begin{align*}
\mathbb{P}\left(\textnormal{there is a particle of } \mathsf{C}^{(s)}\textnormal{ in each interval } (z_i,z_i+dz_i)\right)=\det\left[\mathsf{K}^{(s)}\left(z_i,z_j\right)\right]_{i,j=1}^ldz_1\dots dz_l.
\end{align*}} given in integrable form:
\begin{align*}
\mathsf{K}^{(s)}(x,y)=\frac{1}{2\pi}\frac{\Gamma(s+1)^2}{\Gamma(2s+1)\Gamma(2s+2)}\frac{P^{(s)}(x)Q^{(s)}(y)-P^{(s)}(y)Q^{(s)}(x)}{x-y},
\end{align*}
with
\begin{align*}
P^{(s)}(x)&=2^{2s-\frac{1}{2}}\Gamma\left(s+\frac{1}{2}\right)|x|^{-\frac{1}{2}}J_{s-\frac{1}{2}}\left(\frac{1}{|x|}\right),\\
Q^{(s)}(x)&=\textnormal{sgn}(x)2^{2s+\frac{1}{2}}\Gamma\left(s+\frac{3}{2}\right)|x|^{-\frac{1}{2}}J_{s+\frac{1}{2}}\left(\frac{1}{|x|}\right),
\end{align*}
where $J_s(\cdot)$ is the Bessel function. Let $-\alpha_1^-< -\alpha_2^-< -\alpha_3^-< \cdots < 0$ and $\alpha_1^+> \alpha_2^+> \alpha_3^+> \cdots > 0$ be the corresponding random points of $\mathsf{C}^{(s)}$. We then consider the random variable (this is well defined by the results of Qiu in \cite{Qiu}, see also Theorem \ref{BorodinOlshanskiQiu} below):
\begin{align}\label{DefinitionDisplayRV}
\mathsf{X}(s)=\lim_{m \to \infty} \left[\sum_{i= 1}^{\infty}\alpha_i^+\mathbf{1}\left(\alpha_i^+ >\frac{1}{m^2}\right)-\sum_{i= 1}^{\infty}\alpha_i^-\mathbf{1}\left(\alpha_i^- >\frac{1}{m^2}\right)\right].
\end{align}
We note that the right hand side of (\ref{DefinitionDisplayRV}) depends on the parameter $s$ because the $\alpha_i^{+}$'s and $-\alpha_i^{-}$'s are the points of the determinantal point process $\mathsf{C}^{(s)}$, which is itself depended on $s$ through its explicit correlation kernel $\mathsf{K}^{(s)}$.
\end{defn}
\noindent$\mathsf{X}(s)$ can be thought of as a kind of principal value sum of the points in $\mathsf{C}^{(s)}$. 

We are now in a position to state our main result precisely.

\begin{thm}\label{MainResult}
Let $s>-\frac{1}{2}$ and $0\le h<s+\frac{1}{2}$. Then,
\begin{align}
\lim_{N\to \infty}\frac{1}{N^{s^2+2h}}\mathfrak{F}_N(s,h)\overset{\textnormal{def}}{=}\mathfrak{F}(s,h)=\mathfrak{F}(s,0)2^{-2h}\mathbb{E}\left[|\mathsf{X}(s)|^{2h}\right]
\end{align}
with the limit $\mathfrak{F}(s,h)$ satisfying $0<\mathfrak{F}(s,h)<\infty$. The function $\mathfrak{F}(s,0)$ is given by 
\begin{align*}
\mathfrak{F}(s,0)=\frac{G(s+1)^2}{G(2s+1)},
\end{align*}
where $G$ is the Barnes G-function, given by the infinite product
\begin{align*}
G(1+z)=(2\pi)^{\frac{z}{2}}\exp\left(-\frac{z+z^2(1+\gamma)}{2}\right)\prod_{k=1}^{\infty}\left(1+\frac{z}{k}\right)^k\exp\left(\frac{z^2}{2k}-z\right)
\end{align*}
with $\gamma$ denoting Euler's constant.
\end{thm}

\begin{rmk}
We expect the statement of Theorem \ref{MainResult} still to hold for $-\frac{1}{2}<h<0$. We also discuss the range $-1<s\le -\frac{1}{2}$ in Remark \ref{RemarkInfiniteMeasures}.
\end{rmk}

\subsection{On the leading order coefficient $\mathfrak{F}(s,h)$}\label{SectionLeadingCoefficient}

We now turn our attention to the leading order coefficient $\mathfrak{F}(s,h)$ in the asymptotic. The study of this essentially boils down to analysing the moments $\mathbb{E}\left[|\mathsf{X}(s)|^{2h}\right]$ of the absolute value of the random variable $\mathsf{X}(s)$ since $\mathfrak{F}(s,0)$ is completely explicit. Combining the methods developed in this paper with previous works on this problem \cite{JointMomentsPainleve}, \cite{MixedMoments}, \cite{Winn}, \cite{Dehaye1}, \cite{Hughes} we prove a number of results which shed some light on the distribution of the random variable $\mathsf{X}(s)$ which are of intrinsic interest (see for example Section \ref{SubsectionOrigins} for motivation). We begin with the following formula for the even moments of $\mathsf{X}(s)$ which is proven using probabilistic arguments in Section \ref{SectionPainleveProof}.

\begin{prop}\label{EvenMomentsFormulaProposition}Let $h\in \mathbb{N}$. Denote by $\mathbb{E}_N^{(s)}$ the expectation with respect to the probability measure $\mathsf{M}_N^{(s)}$ on ordered configurations in $\mathbb{R}^N$ given in Definition \ref{DefinitionHuaPickrell}. Then, for $s>h-\frac{1}{2}$ we have:
\begin{align}\label{EvenMomentsFormula}
\mathbb{E}\left[\mathsf{X}(s)^{2h}\right]=\frac{1}{(2h)!}\sum_{k=1}^{2h}(-1)^{2h-k}\binom{2h}{k}\mathbb{E}_k^{(s)}\left[\left(\mathsf{x}_1^{(k)}+\dots+\mathsf{x}_k^{(k)}\right)^{2h}\right].
\end{align}
Moreover, for any $N\ge 1$ the following function, defined for $s>h-\frac{1}{2}$:
\begin{align*}
s\mapsto \mathbb{E}_N^{(s)}\left[\left(\mathsf{x}_1^{(N)}+\dots+\mathsf{x}_N^{(N)}\right)^{2h}\right]
\end{align*}
is a rational function.
\end{prop}
The formula (\ref{EvenMomentsFormula}) above is useful for explicit computations. For example, an easy calculation gives the following expression for the second moment of $\mathsf{X}(s)$:
\begin{align}\label{SecondMoment}
\mathbb{E}\left[\mathsf{X}(s)^{2}\right]=\frac{1}{4s^2-1}, \ \ s>\frac{1}{2}.
\end{align}
Moreover, an immediate consequence of Proposition \ref{EvenMomentsFormulaProposition} is that the explicit rational function expressions, in the variable $s$, for $\frac{\mathfrak{F}(s,h)}{\mathfrak{F}(s,0)}$ computed for $s,h\in \mathbb{N}$ in \cite{Hughes}, \cite{Dehaye1}, \cite{JointMomentsPainleve} extend to non-integer values of $s$, in the range $s>h-\frac{1}{2}$, for fixed $h\in \mathbb{N}$ \footnote{It is claimed in \cite{Dehaye1} that the rational function expressions for $\frac{\mathfrak{F}(s,h)}{\mathfrak{F}(s,0)}$ obtained there for $s,h\in \mathbb{N}$ also hold for non-integer $s$, but the proof of this claim, cf.~Theorem 5.11 therein, does not seem complete to us.  Moreover, it is not clear that the ideas sketched there can be developed to construct a proof.  Simpler formulae for $\frac{\mathfrak{F}_N(s,h)}{\mathfrak{F}_N(s,0)}$ for $s,h\in \mathbb{N}$ and $N\ge 1$ were later obtained in \cite{Dehaye2} and  \cite{JointMomentsPainleve} which do continue to non-integer $s$ by the elementary argument we give later in this paper. Then, one may pass to the $N\to\infty$ limit to obtain a result for the even moments of $\mathsf{X}(s)$ for non-integer $s$.}.

 We then move on to the following corollary of Theorem \ref{MainResult}. We do not see how to verify this result directly from the definition of $\mathsf{X}(s)$ in (\ref{DefinitionDisplayRV}).

\begin{cor}\label{CorollaryPainleve}
Let $s,h\in \mathbb{N}$ with $h\le s$. Then,
\begin{align*}
\mathbb{E}\left[\mathsf{X}(s)^{2h}\right]=2^{2h}(-1)^h\frac{d^{2h}}{dt^{2h}}\left[\exp\int_{0}^{t}\frac{\tau(x)}{x}dx\right]\bigg|_{t=0},
\end{align*}
where $\tau(x)$ is a non-trivial solution to a special case of the $\sigma$-Painlev\'{e} III' equation with two parameters:
\begin{align*}
\left(x\frac{d^2\tau}{dx^2}\right)^2=-4x\left(\frac{d\tau}{dx}\right)^3+(4s^2+4\tau)\left(\frac{d\tau}{dx}\right)^2+x\frac{d\tau}{dx}-\tau,
\end{align*}
with initial conditions:
\begin{align*}
\tau(0)=0, \ \tau'(0)=0.
\end{align*}
Moreover,
\begin{align*}
\mathbb{E}\left[\mathsf{X}(s)^{2h}\right]=(-1)^{\frac{s(s-1)}{2}+h}2^{2h}\frac{G(2s+1)}{G(s+1)^2}\frac{d^{2h}}{dt^{2h}}\left[e^{-\frac{t}{2}}t^{-\frac{s^2}{2}}\det\left(I_{i+j-1}\left(2\sqrt{t}\right)\right)_{i,j=1}^s\right]\bigg|_{t=0}
\end{align*}
where $I_{r}$ is the modified Bessel function.
\end{cor}

\begin{proof}
Direct comparison of the expression for $\mathfrak{F}(s,h)$ from Theorem \ref{MainResult} above with the one from Theorem 2 of \cite{JointMomentsPainleve} and Theorem 1.1 of \cite{MixedMoments} respectively.
\end{proof}

The first part of the result above is actually a facet of a deeper connection between the random variable $\mathsf{X}(s)$ and Painlev\'{e} equations as we show in Proposition \ref{PainleveCharacteristicFunction} below\footnote{By writing $\mathbb{E}\left[e^{\i \frac{t}{2} \mathsf{X}(s)}\right]=\exp\left(\int_{0}^{t}\frac{\mathsf{\Xi}^{(s)}(x)}{x}dx\right)$, differentiating with respect to $t$ and then evaluating at $t=0$ we recover the connection with Painlev\'{e} of Corollary \ref{CorollaryPainleve}.}. We prove Proposition \ref{PainleveCharacteristicFunction} in Section \ref{SectionPainleveProof} by making use of certain results for finite $N$ from \cite{JointMomentsPainleve} and taking the $N\to \infty$ scaling limit using the approach developed in this paper.

\begin{prop}\label{PainleveCharacteristicFunction}
Let $s\in \mathbb{N}_{>1}=\{2,3,4,\dots \}$ and define:
\begin{align}
\mathsf{\Xi}^{(s)}(t)=t\frac{d}{dt}\log \mathbb{E}\left[e^{\i \frac{t}{2} \mathsf{X}(s)}\right].
\end{align} 
Then, there exists $T>0$ such that $\mathsf{\Xi}^{(s)}$ is $C^2$ in $[0,T]$ and $\mathsf{\Xi}^{(s)}$ satisfies a special case of the $\sigma$-Painlev\'{e} III' equation with two parameters:
\begin{align*}
\left(t\frac{d^2\mathsf{\Xi}^{(s)}}{dt^2}\right)^2=-4t\left(\frac{d\mathsf{\Xi}^{(s)}}{dt}\right)^3+(4s^2+4\mathsf{\Xi}^{(s)})\left(\frac{d\mathsf{\Xi}^{(s)}}{dt}\right)^2+t\frac{d\mathsf{\Xi}^{(s)}}{dt}-\mathsf{\Xi}^{(s)},
\end{align*}
with initial conditions:
\begin{align*}
\mathsf{\Xi}^{(s)}(0)=0, \ \frac{d}{dt}\mathsf{\Xi}^{(s)}(t)\big|_{t=0}=0.
\end{align*}

\end{prop}

\begin{rmk}
We expect Proposition \ref{PainleveCharacteristicFunction} to hold for real values of the parameter $s$ as well. If one could extend certain finite $N$ results from \cite{JointMomentsPainleve} to real parameters $s$, then the proof we present here would go through verbatim. However, it appears to be the case that the argument in \cite{JointMomentsPainleve} is strongly dependent on $s$ being an integer at several stages.

A more promising approach might be to attack this problem directly in the limit from the basic definition of $\mathsf{X}(s)$ and make use of the fact that, by the general theory of determinantal point processes, see \cite{Soshnikov}, the characteristic function of $\mathsf{X}(s)$ has a Fredholm determinant representation:
\begin{align*}
\mathbb{E}\left[e^{\i\frac{t}{2}\mathsf{X}(s)}\right]&=\lim_{m\to \infty}\mathbb{E}\left[e^{\i\frac{t}{2}\mathsf{X}_m(s)}\right], \textnormal{ where } \mathsf{X}_m(s)=\sum_{x\in \mathsf{C}^{(s)}}x\mathbf{1}\left(|x|>m^{-2}\right)\\
&=\lim_{m\to \infty}\det \left(\mathbf{1}+\mathsf{K}^{(s)}g_t^{(m)}\right), \textnormal{ where } g_t^{(m)}(z)=e^{\i \frac{t}{2}z\mathbf{1}\left(|z|>m^{-2}\right)}-1.
\end{align*}
There is by now a well developed theory for proving Painlev\'{e} transcendent representations for gap probabilities of determinantal point processes (and certain related generalized quantities). These are also given in terms of Fredholm determinants and a typical example is of the form (for some correlation kernel $\mathsf{K}$):
\begin{align*}
\det \left(\mathbf{1}+\mathsf{K}f_t\right), \textnormal{ where } f_t(z)=-\mathbf{1}\left(z>t\right).
\end{align*}
With regards to the history of this family of problems, it is important to mention the pioneering work of Jimbo, Miwa, M\^{o}ri and Sato \cite{JimboMiwaMoriSato} and also the four systematic approaches due to Tracy and Widom \cite{TracyWidom}, Adler and Van Moerbeke \cite{AdlerVanMoerbeke}, Forrester and Witte \cite{ForresterWitte} and Borodin and Deift \cite{BorodinDeift}. However, it appears that our problem does not fall into any of these frameworks and it would be interesting to investigate it further as this would likely lead to analogous results for a wider class of determinantal point processes (as is the case for gap probabilities).
\end{rmk}

Finally, it is also possible to obtain some rather striking expressions for $\mathbb{E}\left[|\mathsf{X}(s)|^{2h}\right]$ for integer $s$ and  half-integer $h$ in terms of certain combinatorial sums, see \cite{Winn}. To state these results would require us to introduce a considerable amount of new notation and we refrain from doing this here; we instead invite the interested reader simply to compare the expression for $\mathfrak{F}(s,h)$ above with the corresponding expressions from \cite{Winn}. We just state here the simplest possible instance of such a formula, that follows by combining our results with those of Winn \cite{Winn}, since this is also one of the few cases proven in the number theoretic setting \cite{ConreyGhosh}, see Section \ref{NumberTheorySection}. Again, we do not see how to verify this result directly from the definition of $\mathsf{X}(s)$ in (\ref{DefinitionDisplayRV}).
\begin{cor}
We have:
\begin{align*}
\mathbb{E}\left[|\mathsf{X}(1)|\right]=\frac{e^2-5}{2\pi}.
\end{align*}
\end{cor}
\begin{proof}
Compare the expression for $\mathfrak{F}\left(1,\frac{1}{2}\right)$ above with the one from \cite{Winn}.
\end{proof}

\subsection{Number theoretic consequences}\label{NumberTheorySection}

We now explain the number-theoretic consequences of our work following the philosophy of \cite{KeatingSnaith}, \cite{KeatingSnaithLfunctions}, \cite{ConreyRubinsteinSnaith}. Consider the Riemann zeta function $\zeta$ given by:
\begin{align*}
\zeta(z)=\sum_{n=1}^{\infty}\frac{1}{n^z}, \ z=\sigma+\i t, \sigma >1
\end{align*}
which can be extended to the whole complex $z$-plane through the functional equation:
\begin{align*}
\xi(z)=\frac{1}{2}z(z-1)\pi^{-\frac{z}{2}}\Gamma\left(\frac{z}{2}\right)\zeta(z)=\xi(1-z).
\end{align*}
Also define Hardy's function:
\begin{align*}
\mathcal{Z}(t)=\frac{\pi^{-\frac{\i t}{2}}\Gamma\left(\frac{1}{4}+\frac{\i t}{2}\right)}{\big|\Gamma\left(\frac{1}{4}+\frac{\i t}{2}\right)\big|}\zeta\left(\frac{1}{2}+\i t\right)
\end{align*}
and observe that $|\mathcal{Z}(t)|=\big|\zeta\left(\frac{1}{2}+\i t\right)\big|$. By the work of \cite{KeatingSnaith}, \cite{KeatingSnaithLfunctions} \cite{Hughes}, the characteristic polynomial $Z_{\mathbf{U}}$ can be understood as the random matrix analogue of the Riemann zeta function $\zeta$ while $V_{\mathbf{U}}$ is the analogue of Hardy's function $\mathcal{Z}$. Work of Keating and Snaith \cite{KeatingSnaith}, \cite{KeatingSnaithLfunctions}, Hall \cite{Hall} and Hughes \cite{Hughes} culminated in the following conjecture:
\begin{align*}
\frac{1}{T}\int_{0}^{T}\big|\zeta\left(\frac{1}{2}+\i t\right)\big|^{2s-2h}\big|\zeta'\left(\frac{1}{2}+\i t\right)\big|^{2h}dt\sim \mathfrak{C}_{\zeta}\left(s,h\right)\left(\log T\right)^{s^2+2h},
\end{align*}
for a certain (unknown in the most general case when both $s$ and $h$ are arbitrary real numbers) constant $\mathfrak{C}_{\zeta}\left(s,h\right)$. For integer $s$ and $h$ it can be shown, see for example Section 6.3 in \cite{Hughes}, that this conjecture is equivalent to:
\begin{align*}
\frac{1}{T}\int_{0}^{T}|\mathcal{Z}(t)|^{2s-2h}|\mathcal{Z}'(t)|^{2h}dt\sim \mathfrak{C}_{\mathcal{Z}}\left(s,h\right)\left(\log T\right)^{s^2+2h}
\end{align*}
and $\mathfrak{C}_{\mathcal{Z}}\left(s,h\right)$ can be related explicitly to $\mathfrak{C}_{\zeta}\left(s,h\right)$. The conjecture has been proven for $s=1,2$ and integer $h$. For $s=1$, Hardy and Littlewood \cite{HardyLittlewood} proved $\mathfrak{C}_{\zeta}\left(1,0\right)=1$ and then Ingham \cite{Ingham} showed $\mathfrak{C}_{\zeta}\left(1,1\right)=\frac{1}{3}$. For $s=2$, we have $\mathfrak{C}_{\zeta}(2,0)=\frac{1}{2\pi^2}$ from \cite{Ingham} due to Ingham, and moreover $\mathfrak{C}_{\zeta}(1,2)=\frac{1}{15\pi^2}$ and $\mathfrak{C}_{\zeta}(2,2)=\frac{61}{1680\pi^2}$ which are both due to Conrey from \cite{Conrey}. The analogous results for $\mathfrak{C}_{\mathcal{Z}}$ can be easily deduced, see for example Section 6.3 in \cite{Hughes}. Finally, the conjecture for $\mathcal{Z}$ was proven (assuming the Riemann hypothesis) for $s=1$ and $h=\frac{1}{2}$ by Conrey and Ghosh \cite{ConreyGhosh} giving $\mathfrak{C}_{\mathcal{Z}}\left(1,\frac{1}{2}\right)=\frac{e^2-5}{4\pi}$, but $\mathfrak{C}_{\zeta}\left(1,\frac{1}{2}\right)$ is still unknown.

Based on the results of this paper and the philosophy developed in \cite{KeatingSnaith}, \cite{KeatingSnaithLfunctions}, \cite{Hughes} we then conjecture the following for the joint moments of $\mathcal{Z}$ and its derivative. This gives the first representation for the constant $\mathfrak{C}_{\mathcal{Z}}(s,h)$ for general real values of $s$ and $h$.

\begin{conj} Let $s>-\frac{1}{2}$ and $0\le h<s+\frac{1}{2}$. Then,
\begin{align}
\frac{1}{T}\int_{0}^{T}|\mathcal{Z}(t)|^{2s-2h}|\mathcal{Z}'(t)|^{2h}dt\sim \mathsf{a}(s) \mathfrak{F}(s,h)\left(\log T\right)^{s^2+2h}=\mathsf{a}(s) \mathfrak{F}(s,0)2^{-2h}\mathbb{E}\left[|\mathsf{X}(s)|^{2h}\right] \left(\log T\right)^{s^2+2h},
\end{align}
where the arithmetic constant $\mathsf{a}(s)$ is given by:
\begin{align*}
\mathsf{a}(s)=\prod_{p \ \textnormal{prime}}^{} \left(1-\frac{1}{p}\right)^{s^2}\sum_{m=0}^{\infty}\left(\frac{\Gamma(m+s)}{m!\Gamma(s)}\right)^2p^{-m}.
\end{align*}
\end{conj}

\begin{rmk}
The $h=0$ case is from the work of Keating and Snaith \cite{KeatingSnaith}. Moreover, for an explanation in the case $h=0$ of why the moments factor asymptotically into an arithmetic term $\mathsf{a}(s)$ and a random matrix term $\mathfrak{F}(s,0)$, see \cite{GonekHughesKeating}.
\end{rmk}

\subsection{The origins of the random variable $\mathsf{X}(s)$}\label{SubsectionOrigins}

The origins of the random variable $\mathsf{X}(s)$ lie in the classification and study of ergodic measures for actions of inductively compact groups:
\begin{align*}
\mathbb{G}(1)\subset \mathbb{G}(2) \subset \cdots \subset \mathbb{G}(N) \subset \cdots
\end{align*}
on some Polish space $\mathbb{X}$, see \cite{VershikErgodic}, \cite{OlshanskiVershik}. 

Many classical results from probability and ergodic theory fall into this abstract framework, the simplest example being de Finetti's theorem. Here, we are interested in the infinite dimensional unitary group $\mathbb{U}(\infty)=\underset{\rightarrow}{\lim}\mathbb{U}(N)$ \footnote{This is the inductive limit group under the mappings:
\begin{align*}
\mathbf{U}\mapsto \begin{bmatrix}
\mathbf{U} & 0\\
0 & 1
\end{bmatrix}.
\end{align*}} acting on the space of infinite Hermitian matrices $\mathbb{H}=\underset{\leftarrow}{\lim}\mathbb{H}(N)$ \footnote{This is the projective limit under the maps $\mathsf{\Pi}_N^{N+1}:\mathbb{H}(N+1)\to \mathbb{H}(N)$, where $\mathbb{H}(N)$ is the space of $N\times N$ Hermitian matrices:
\begin{align*}
\mathsf{\Pi}_N^{N+1}\left(\{\mathbf{H}_{ij}\}_{i,j=1}^{N+1}\right)=\{\mathbf{H}_{ij}\}_{i,j=1}^{N}. 
\end{align*}} by conjugation, namely for each $\mathbf{U}\in \mathbb{U}(\infty)$ we have a map $\mathsf{T}_{\mathbf{U}}:\mathbb{H}\to \mathbb{H}$ given by $\mathsf{T}_{\mathbf{U}}\left(\mathbf{H}\right)=\mathbf{U}^*\mathbf{H}\mathbf{U}$. 

The ergodic measures for this action on $\mathbb{H}$ have been classified by Pickrell \cite{Pickrell} and Olshanski and Vershik \cite{OlshanskiVershik}. This classification problem is in fact equivalent to the classification of totally positive functions due to Schoenberg, see \cite{Schoenberg}, \cite{Pickrell}, \cite{OlshanskiVershik}. It is furthermore equivalent to the classification of spherical representation of the infinite dimensional Cartan motion group $\mathbb{U}(\infty)\ltimes\mathbb{H}(\infty)=\underset{\rightarrow}{\lim}\mathbb{U}(N)\ltimes\mathbb{H}(N)$ \footnote{Elements in $\mathbb{U}(N)\ltimes\mathbb{H}(N)$ are the pairs $\left(\mathbf{U},\mathbf{H}\right)\in \mathbb{U}(N)\times\mathbb{H}(N)$ with the multiplication rule:
\begin{align*}
\left(\mathbf{U}_1,\mathbf{H}_1\right)\bullet \left(\mathbf{U}_2,\mathbf{H}_2\right)=\left(\mathbf{U}_1\mathbf{U}_2,\mathbf{U}_2^*\mathbf{H}_1\mathbf{U}_2+\mathbf{H}_2\right).
\end{align*}
The inductive limit is taken with respect to the mappings:
\begin{align*}
\left(\mathbf{U},\mathbf{H}\right)\mapsto \left(\begin{bmatrix}
\mathbf{U} & 0 \\
0 & 1
\end{bmatrix}, \begin{bmatrix}
\mathbf{H} & 0\\
0 & 0
\end{bmatrix}\right).
\end{align*}}, see \cite{OlshanskiVershik} for the details. It also has a probabilistic interpretation in terms of determining the boundary of a certain Markov chain see \cite{BorodinOlshanski}, \cite{HuaPickrellDiffusions} and \cite{OrbitalBetaBoundary} for a generalization.  Finally, closely related problems of representation theoretic origin, that essentially belong to the same circle of ideas, include the classification of irreducible characters of the infinite dimensional unitary and symmetric groups, see for example \cite{VershikKerov}, \cite{VershikKerovSymmetric}, \cite{OkounkovOlshanskiJack}, \cite{KerovOkounkovOlshanski}.

The result of Pickrell and Olshanski-Vershik reads as follows. The ergodic $\mathbb{U}(\infty)$-invariant measures on $\mathbb{H}$ are in bijection with the infinite dimensional space $\Omega$ defined by:
\begin{align*}
\Omega=\big\{\left(\{\alpha_i^+ \}_{i=1}^\infty, \{\alpha_i^- \}_{i=1}^\infty, \gamma_1,\gamma_2\right)\subset \mathbb{R}^{2\infty+2}; \ \alpha_1^+\ge \alpha_2^+\ge \alpha_3^+\ge \cdots\ge 0; \ \alpha_1^-\ge \alpha_2^-\ge \alpha_3^- \ge \cdots \ge 0,\\
\textnormal{with } \sum_{i=1}^{\infty}\left(\alpha_i^+\right)^2+\left(\alpha_i^-\right)^2<\infty, \ \gamma_2\ge 0\big\}
\end{align*}
and each ergodic measure $\mathfrak{N}_{\omega}$ is determined by its Fourier transform which has an explicit expression, see \cite{OlshanskiVershik}, \cite{BorodinOlshanski}, \cite{OrbitalBetaBoundary}. Moreover, Borodin and Olshanski proved in \cite{BorodinOlshanski} that for any $\mathbb{U}(\infty)$-invariant probability measure $\mathfrak{M}$ on $\mathbb{H}$, there exists a unique probability measure $\nu^{\mathfrak{M}}$ on $\Omega$ such that:
\begin{align*}
\mathfrak{M}\left(d\mathbf{H}\right)=\int_{\Omega}^{}\nu^{\mathfrak{M}}(d\omega) \mathfrak{N}_{\omega}\left(d\mathbf{H}\right).
\end{align*}
Thus, a natural direction of research would be to describe explicitly the measure $\nu^{\mathfrak{M}}$ governing the decomposition into ergodic components for some distinguished $\mathbb{U}(\infty)$-invariant probability measure $\mathfrak{M}$. In the representation theoretic setting this is known as the problem of harmonic analysis, see for example \cite{KerovOlshanskiVershik}, \cite{BorodinOlshanskiHarmonic}, \cite{OlshanskiHarmonic}, \cite{GorinOlshanski}.

A distinguished example of such $\mathbb{U}(\infty)$-invariant probability measures on $\mathbb{H}$ are the so-called Hua-Pickrell measures $\mathfrak{M}^{(s)}$ depending on a complex parameter $s$ (in this paper we will restrict our attention to real values of $s$) and here we denote by $\nu^{(s)}$ the corresponding measure $\nu^{\mathfrak{M}^{(s)}}$ of ergodic decomposition. These measures were essentially introduced by Hua in his book \cite{Hua} on harmonic analysis on matrix spaces and are constructed from their finite dimensional projections $\mathfrak{M}_N^{(s)}$ on $\mathbb{H}(N)$ which are also known in random matrix theory as the Cauchy ensemble and under the Cayley transform as the circular Jacobi ensemble on $\mathbb{U}(N)$. We refer the reader to the extensive literature for properties of these \cite{Hua}, \cite{Neretin}, \cite{BorodinOlshanski}, \cite{Qiu}, \cite{BufetovQiu}, \cite{BourgadeNajnudelNikeghbali}, \cite{BourgadeNikeghbaliRouault} and closely related measures \cite{Pickrellmeasure}, \cite{BufetovI}, \cite{BufetovII}, \cite{BufetovIII}, \cite{AssiotisInverseWishart}, and connections ranging from Painlev\'{e} equations \cite{ForresterWitteCauchy} to stochastic processes \cite{HuaPickrellDiffusions}, \cite{MatrixBougerol}.

The study of $\nu^{(s)}$ was initiated by Borodin and Olshanski in \cite{BorodinOlshanski} who determined the law of the parameters $\left(\{ \alpha_i^+\}_{i=1}^{\infty},\{\alpha_i^- \}_{i=1}^{\infty}\right)$ by proving that, under $\nu^{(s)}$:
\begin{align*}
\{\alpha_i^+ \} \sqcup \{-\alpha_i^-\} \overset{\textnormal{d}}{=} \mathsf{C}^{(s)}
\end{align*}
where $\mathsf{C}^{(s)}$ is the determinantal point process from Definition \ref{DefinitionGamma1}. The law of the parameters $\left(\gamma_1,\gamma_2\right)$ was left as an open question for several years until in a breakthrough work Qiu \cite{Qiu} proved that $\nu^{(s)}$ almost surely
$\gamma_2\equiv 0$ and that moreover under $\nu^{(s)}$:
\begin{align*}
\gamma_1\overset{\textnormal{d}}{=}\mathsf{X}(s).
\end{align*}
Thus, our results from Section \ref{SectionLeadingCoefficient} on the random variable $\mathsf{X}(s)$ are of independent interest since in particular the parameter $\gamma_1$ is arguably the hardest parameter to study, not only in the problem of ergodic decomposition of $\mathfrak{M}^{(s)}$ on $\Omega$, but in essentially all the other allied problems on possibly different spaces we have alluded to \cite{BorodinOlshanskiHarmonic}, \cite{KerovOlshanskiVershik}, \cite{GorinOlshanski}, \cite{BufetovI}, \cite{BufetovII}, \cite{BufetovIII}, \cite{AssiotisInverseWishart}.

\subsection{The strategy of proof}\label{SubsectionStrategyProof}
The proof of Theorem \ref{MainResult} rests upon two main ideas which can be viewed as miracles of the integrability underlying the problem. Armed with these two insights the need for lengthy and complicated computations that were required to establish the conjecture for special cases in previous works \cite{Hughes}, \cite{ConreyRubinsteinSnaith}, \cite{Dehaye1}, \cite{Dehaye2}, \cite{Winn}, \cite{JointMomentsPainleve}, \cite{MixedMoments} disappears entirely.

The first key ingredient that we prove is a representation of $\mathfrak{F}_N(s,h)$ in terms of $\mathfrak{F}_N(s,0)$, which has an explicit expression in terms of special functions amenable to asymptotic analysis, and the following moments of a certain average $\mathbb{E}\left[\big|\sum_{i=1}^{N}\frac{\mathsf{x}^{(N)}_i}{N}\big|^{2h}\right]$. Here, the random variables $(\mathsf{x}_1^{(N)},\dots,\mathsf{x}_N^{(N)})$ have the same distribution as the non-increasing eigenvalues of a random Hermitian matrix with law $\mathfrak{M}_N^{(s)}$. 

From the important work of Qiu \cite{Qiu} it is known that:
\begin{align*}
\frac{1}{N}\sum_{i=1}^{N}\mathsf{x}_i^{(N)}\overset{\textnormal{d}}{\longrightarrow} \mathsf{X}(s), \ \textnormal{ as } N \to \infty.
\end{align*}
Then, in order to upgrade this result to convergence of the moments:
\begin{align*}
\mathbb{E}\left[\bigg|\sum_{i=1}^{N}\frac{\mathsf{x}^{(N)}_i}{N}\bigg|^{2h}\right]\longrightarrow\mathbb{E}\left[|\mathsf{X}(s)|^{2h}\right], \ \textnormal{ as } N \to \infty,
\end{align*}
one needs to prove uniform integrability or, as we do here, show uniform boundedness for some higher moment.

It is well known that $(\mathsf{x}_1^{(N)},\dots,\mathsf{x}_N^{(N)})$ have a very special structure, namely it is a standard fact from random matrix theory that they give rise to a determinantal point process $\mathsf{C}_N^{(s)}$ with $N$ points on $\mathbb{R}$ with correlation kernel $\mathsf{K}_N^{(s)}$ given in terms of the Pseudo-Jacobi polynomials. Then, it is possible to expand the $2m$-th moment of the average we are interested in for $m\in \mathbb{N}$ in terms of integrals involving the correlation functions of the point process $\mathsf{C}_N^{(s)}$ up to order $2m$. On the other hand, from the results of \cite{BorodinOlshanski} we know that:
\begin{align*}
\left(\mathsf{x}_1^{(N)}\ge \cdots \ge \mathsf{x}_N^{(N)}\right)\sim \left(N\alpha_1^+,N \alpha_2^+ ,\dots,-N\alpha_2^-,-N\alpha_1^-\right)
\end{align*}
where the alpha parameters satisfy $\{\alpha_i^+ \} \sqcup \{-\alpha_i^-\} \overset{\textnormal{d}}{=} \mathsf{C}^{(s)}$ and this suggests, and in fact it can be shown\footnote{Using the uniform estimates from \cite{Golinskii} and some technical work it is possible to show that:
\begin{align*}
\mathbb{E}\left[\sum_{i=1}^{N}\frac{|\mathsf{x}_i^{(N)}|}{N}\right]\sim \textnormal{c} \log N.
\end{align*}}, that the averages that we want to control uniformly in $N$ do not converge if we bring the absolute values inside. This was also one of the main difficulties that had to be resolved in the analysis by Qiu in \cite{Qiu} which involved considering the second moment ($m=1$). Thus, it is essential that a cancellation due to symmetry around the origin of the points in $\mathsf{C}_N^{(s)}$ is taken into account.

Our initial approach, which we were able to make work after significant effort, was to firstly estimate integrals of correlation functions of any order in terms of integrals of the first correlation function, while taking into account the essential cancellations due to symmetry, which turned out to be quite tricky\footnote{Without having to take the cancellation into account, obtaining an estimate in terms of the first correlation function is trivial by a straightforward application of Hadamard's inequality.}. Then, for the integrals of the first correlation function $\rho_{N,1}^{(s)}(x)=\mathsf{K}_N^{(s)}(x,x)$ it is possible to obtain some precise uniform estimates that are sufficient for our purposes by using certain delicate uniform estimates due to Golinskii \cite{Golinskii} for the Pseudo-Jacobi polynomials, or more precisely for their unit circle analogues. Although this approach was successful, it had an intrinsic limitation in that it was not possible to access the range $h\in [s,s+\frac{1}{2})$
which includes the distinguished case $s=h$ and in addition $s$ had to satisfy $s>\frac{1}{2}$ which in turn misses the other distinguished case $s=0$.

We were able to circumvent these difficulties using a second key insight, which is based on the elementary observation that the sum of the eigenvalues of a matrix is equal to its trace, to which we turn our attention. Due to the remarkable property of consistency of the Hua-Pickrell measures $\mathfrak{M}_N^{(s)}$, for all $N\ge 1$ the diagonal elements of such a random matrix turn out to be exchangeable (they are highly non-trivially correlated, far from being independent) identically distributed random variables with the Pearson IV distribution. In particular, they do not grow with $N$ as the eigenvalues $(\mathsf{x}_1^{(N)},\dots,\mathsf{x}_N^{(N)})$ do.
From this result, for $s>0$, uniform boundedness for the moments of interest follows by elementary means and gives us Theorem \ref{MainResult}. In order to treat the range $-\frac{1}{2}<s\le 0$, however, we need to go deeper into the exchangeable structure of the diagonal elements and connect this back to the random variable $\mathsf{X}(s)$ and point process $\mathsf{C}^{(s)}$, which we analyse in detail. The reader is  referred to the brief discussion before Lemma \ref{LowMomentsConvLem} for an explanation of the significance of the two different ranges $s>0$ and $-\frac{1}{2}<s\le 0$. It is also worth pointing out that the fact that the diagonal entries have a special structure, namely exchangeability, also plays an implicit, although important, role in the proof of classification of Olshanski and Vershik \cite{OlshanskiVershik}. Moreover, diagonal elements take centre stage in the generalization of the classification to $\beta$-ensembles in \cite{OrbitalBetaBoundary}.

Finally, it is worth mentioning that connections between moments of traces of random matrices from the classical ensembles and hypergeometric orthogonal polynomials, enumerative combinatorics and integrable systems have recently been established in \cite{CundenMezzadriOConnellSimm}, \cite{CundenDahlqvistOConnell}, \cite{GisonniGravaRuzza}. We believe that analogous results should exist for the Hua-Pickrell measures and are currently investigating this.

\paragraph{Acknowledgements} TA thanks Neil O'Connell and Mo Dick Wong for useful discussions. The research described here was supported by ERC Advanced Grant  740900 (LogCorRM). JPK also acknowledges support from a Royal Society Wolfson Research Merit Award. Research of JW supported by ERC Advanced Grant 669306 (IntRanSt). We are grateful to anonymous referees for a careful reading of the paper and for useful comments and suggestions.

\section{Preliminaries and proof of the main result}

\subsection{Preliminaries on the Hua-Pickrell measures}
We begin with a number of definitions and preliminary results. Let $\mathbb{W}_N$ denote the Weyl chamber:
\begin{align*}
\mathbb{W}_N=\{\mathbf{x}=(x_1,x_2,\dots,x_N)\in \mathbb{R}^N:x_1\ge x_2\ge \cdots \ge x_N \}.
\end{align*}
We then define the Hua-Pickrell measures\footnote{The corresponding Hua-Pickrell measures $\mathfrak{M}^{(s)}_N$ on $N\times N$ Hermitian matrices $\mathbb{H}(N)$ will be defined precisely and used in the proof of Proposition \ref{DiagonalConsistency} later on.} \cite{Hua}, \cite{Pickrellmeasure}, \cite{NeretinDuke} on $\mathbb{W}_N$ which will be at the centre of our argument.

\begin{defn}\label{DefinitionHuaPickrell}
Let $N\ge 1$ and $s>-\frac{1}{2}$. Define the Hua-Pickrell probability measure $\mathsf{M}_N^{(s)}$ on $\mathbb{W}_N$ given by:
\begin{align*}
\mathsf{M}_N^{(s)}(d\mathbf{x})=\frac{1}{\mathsf{c}_N^{(s)}}\times\prod_{j=1}^{N}\frac{1}{(1+x^2_j)^{N+s}}\Delta_N(\mathbf{x})^2dx_1\cdots dx_N
\end{align*}
where $\Delta_N(\mathbf{x})$ is the Vandermonde determinant:
\begin{align*}
\Delta_N(\mathbf{x})=\prod_{1\le i < j\le N}^{}(x_j-x_i)
\end{align*}
and the normalization constant is given by (see \cite{Neretin}, \cite{Forrester}):
\begin{align*}
\mathsf{c}_N^{(s)}=(2\pi)^N 2^{-N^2-2sN} G(N+1)\prod_{j=1}^{N}\frac{\Gamma(2s+N-j+1)}{\Gamma(s+N-j+1)^2}
\end{align*}
and $G$ is the Barnes G-function.
\end{defn}

We have the following important theorem, that will be used as a key input in our argument. This is a combination of results of Borodin and Olshanski \cite{BorodinOlshanski} and Qiu \cite{Qiu}.

\begin{thm}\label{BorodinOlshanskiQiu}
Let $s>-\frac{1}{2}$. Then,
\begin{align}
\frac{1}{N}\sum_{i=1}^{N}\mathsf{x}_i^{(N)}\overset{\textnormal{d}}{\longrightarrow} \mathsf{X}(s), \ \textnormal{ as } N \to \infty,
\end{align}
where $(\mathsf{x}_1^{(N)},\dots,\mathsf{x}_N^{(N)})$ has law $\mathsf{M}_N^{(s)}$ and $\mathsf{X}(s)$ is the random variable from Definition \ref{DefinitionGamma1}.
\end{thm}

\begin{rmk}
The statement for the convergence of the average of $(\mathsf{x}_1^{(N)},\dots,\mathsf{x}_N^{(N)})$ to some random variable is due to Borodin and Olshanski \cite{BorodinOlshanski}, see Section 5 therein, also Section 2.1 in \cite{Qiu}. The explicit identification of the limiting random variable with $\mathsf{X}(s)$, along with many other interesting results, is due to Qiu \cite{Qiu}, see Theorems 1.2, 1.3 and in particular Theorem 2.3 therein.
\end{rmk}

\subsection{Preliminaries on the joint moments}

Returning to moments of characteristic polynomials, we have the following result of Keating and Snaith \cite{KeatingSnaith} which gives an explicit expression for $\mathfrak{F}_N(s,0)$. This makes use of the Weyl integration formula, which in this setting states that for $H$ a class function on $\mathbb{U}(N)$, namely a function depending only on the eigenvalues $e^{\i \theta_1},\dots,e^{\i \theta_N}$ of $\mathbf{U}\in \mathbb{U}(N)$, we have: 
\begin{align*}
\int_{\mathbb{U}(N)}^{}H(\mathbf{U})d\mu_N\left(\mathbf{U}\right)=\frac{1}{(2\pi)^NN!}\int_{0}^{2\pi}\cdots \int_{0}^{2\pi}H(\mathbf{U})\prod_{1\le j<k \le N}^{}|e^{\i \theta_k}-e^{\i \theta_j}|^2d\theta_1\dots d\theta_N.
\end{align*}
A trigonometric substitution then turns the resulting expression for $\mathfrak{F}_N(s,0)$ into a version of the celebrated Selberg integral, see \cite{Forrester}, which has an explicit evaluation\footnote{On a more conceptual level the computation just described is a shadow of an application of the (inverse) Cayley transform which takes the circular Jacobi ensemble on $\mathbb{U}(N)$ to the Hua-Pickrell measures on $\mathbb{H}(N)$.}. From this the limiting result follows rather easily from the asymptotics of the Barnes G-function.
\begin{thm}\label{KeatingSnaith}
For $s>-\frac{1}{2}$ we have:
\begin{align*}
\mathfrak{F}_N(s,0)=\frac{G(N+2s+1)G(N+1)G(s+1)^2}{G(N+s+1)^2G(2s+1)}
\end{align*}
where again $G$ is the Barnes G-function. Moreover,
\begin{align*}
\lim_{N\to \infty}\frac{1}{N^{s^2}}\mathfrak{F}_N(s,0)=\mathfrak{F}(s,0)=\frac{G(s+1)^2}{G(2s+1)}.
\end{align*}
\end{thm}

Moreover, we have the following integral representation for $\mathfrak{F}_N(s,h)$. The computation is completely analogous to the one required for Theorem \ref{KeatingSnaith}, using Weyl's integration formula. As far as we are aware, for general $h\neq 0$, there is no exact expression in terms of special functions for $\mathfrak{F}_N(s,h)$ from which the asymptotics can be readily established as for $\mathfrak{F}_N(s,0)$ in Theorem \ref{KeatingSnaith}\footnote{Of course, as already mentioned in the introduction, for integer $s$ and $h$ a zoo of explicit expressions for $\mathfrak{F}_N(s,h)$ is known, from contour integrals to Painlev\'{e} equations and combinatorial sums see \cite{ConreyRubinsteinSnaith}, \cite{JointMomentsPainleve}, \cite{MixedMoments}, \cite{Winn}, \cite{Dehaye1}. However, all of these require substantial efforts to establish the asymptotic.}. Nevertheless, Proposition \ref{Winn} will be sufficient for our purposes.

\begin{prop}\label{Winn}
For $-\frac{1}{2}<h<s+\frac{1}{2}$ we have:
\begin{align}
\mathfrak{F}_N(s,h)&=\frac{2^{N^2+2sN-2h}}{(2\pi)^N N!}\int_{-\infty}^{\infty}\cdots \int_{-\infty}^{\infty} \prod_{j=1}^{N}\frac{1}{(1+x^2_j)^{N+s}}|x_1+\cdots+x_N|^{2h}\Delta_N(\mathbf{x})^2dx_1\cdots dx_N\\
&=\frac{2^{N^2+2sN-2h}}{(2\pi)^N }\int_{(x_1,\dots,x_N)\in \mathbb{W}_N}^{} \prod_{j=1}^{N}\frac{1}{(1+x^2_j)^{N+s}}|x_1+\cdots+x_N|^{2h}\Delta_N(\mathbf{x})^2dx_1\cdots dx_N.\label{IntegralRepresentation}
\end{align}
\end{prop}

\begin{proof}
The first equality is exactly Proposition 2 in \cite{Winn}, while the second one follows by symmetry.
\end{proof}

We finally make the following simple but important observation:
\begin{lem}\label{NormalizationLemma}
Let $s>-\frac{1}{2}$. Then,
\begin{align}\label{NormalizationObservation}
\mathsf{c}_N^{(s)}=(2\pi)^N 2^{-N^2-2sN}  \mathfrak{F}_N(s,0).
\end{align}
\end{lem}
\begin{proof}
Immediate from Proposition \ref{Winn} by putting $h=0$ and using the fact that $\mathsf{M}_N^{(s)}$ is a probability measure. Alternatively it can be established by a direct computation making use of the explicit formulae for $\mathfrak{F}_N(s,0)$ and $\mathsf{c}_N^{(s)}$ and the functional equation for the Barnes G-function: $G(1+z)=\Gamma(z)G(z)$.
\end{proof}

\subsection{Joint moments via the Hua-Pickrell measures}

The next key observation is at the heart of our work:

\begin{prop}\label{HuaPickrellRepresentation}
For $s>-\frac{1}{2}$ and $-\frac{1}{2}<h<s+\frac{1}{2}$ we have:
\begin{align*}
\mathfrak{F}_N(s,h)=\mathfrak{F}_N(s,0)2^{-2h}\mathbb{E}^{(s)}_N\left[\bigg|\sum_{i=1}^{N}\mathsf{x}^{(N)}_i\bigg|^{2h}\right],
\end{align*}
where $\mathbb{E}_N^{(s)}$ denotes the expectation with respect to $\mathsf{M}_N^{(s)}$.
\end{prop}

\begin{proof}
Multiply and divide display (\ref{IntegralRepresentation}) in Proposition \ref{Winn} by $\mathsf{c}_N^{(s)}$ and use Lemma \ref{NormalizationLemma}.
\end{proof}

\begin{rmk}\label{RemarkComplexExtension1}
Proposition 2 in \cite{Winn} can be extended to complex parameters $s,h$ such that $-\frac{1}{2}<\Re(h)<\Re(s)+\frac{1}{2}$, see display (3.20) in \cite{Winn}.  Moreover, the Hua-Pickrell measures $\mathsf{M}_N^{(s)}$ also have a natural generalization to a complex parameter $s$ such that $\Re(s)>-\frac{1}{2}$, see \cite{NeretinDuke}, \cite{BorodinOlshanski}. However, we note that these more general measures from \cite{NeretinDuke}, \cite{BorodinOlshanski} are not given by analytically extending in the parameter $s$ the Definition \ref{DefinitionHuaPickrell} (if this is done the resulting quantity is in general no longer positive; it is only a complex measure and so all probabilistic structures used in the sequel are lost). For this reason we cannot simply extend, analytically in $s$, the equality in Proposition \ref{HuaPickrellRepresentation} because the expectation with respect to $\mathsf{M}_N^{(s)}$ on the right hand side does not extend in such fashion.

Finally, for real $s>-\frac{1}{2}$ and $h\in \mathbb{C}$ such that $-\frac{1}{2}<\Re(h)<s+\frac{1}{2}$ the representation in Proposition \ref{HuaPickrellRepresentation} still holds verbatim.
\end{rmk}
\begin{rmk}\label{RemarkInfiniteMeasures}
It might also be possible to give a representation for $\mathfrak{F}_N(s,h)$ for $-1<s \le -\frac{1}{2}$ and $-\frac{1}{2}<h<s+\frac{1}{2}$ in terms of the so-called infinite Hua-Pickrell measures (which can no longer be normalized to be probability measures, see \cite{Qiu}, \cite{BufetovI}, \cite{BufetovII}, \cite{BufetovIII}). We do not pursue this here but it would be interesting to explore further.
\end{rmk}

\begin{prop}\label{ConvergenceProposition}
Let $s>-\frac{1}{2}$ and $0\le h<s+\frac{1}{2}$. Assume the sequence of random variables:
\begin{align}\label{RandomVariables}
\bigg\{ \bigg|\sum_{i=1}^{N}\frac{\mathsf{x}^{(N)}_i}{N}\bigg|^{2h} \bigg\}_{N\ge 1}
\end{align}
with $(\mathsf{x}_1^{(N)},\dots,\mathsf{x}_N^{(N)})$ having law $\mathsf{M}_N^{(s)}$, is uniformly integrable. Then,
\begin{align}
\lim_{N\to \infty}\frac{1}{N^{s^2+2h}}\mathfrak{F}_N(s,h)=\mathfrak{F}(s,0)2^{-2h}\mathbb{E}\left[|\mathsf{X}(s)|^{2h}\right].
\end{align}
\end{prop}

\begin{proof}
From Proposition \ref{HuaPickrellRepresentation} and Theorem \ref{KeatingSnaith} we need to show:
\begin{align*}
\mathbb{E}^{(s)}_N\left[\bigg|\sum_{i=1}^{N}\frac{\mathsf{x}^{(N)}_i}{N}\bigg|^{2h}\right]\overset{N \to \infty}{\longrightarrow}\mathbb{E}\left[|\mathsf{X}(s)|^{2h}\right].
\end{align*}
Then, making use of Theorem \ref{BorodinOlshanskiQiu} it moreover suffices, see Lemma 4.11 in \cite{Kallenberg}, to show that the sequence of random variables $\bigg\{ \bigg|\sum_{i=1}^{N}\frac{\mathsf{x}^{(N)}_i}{N}\bigg|^{2h} \bigg\}_{N\ge 1}$, with $(\mathsf{x}_1^{(N)},\dots,\mathsf{x}_N^{(N)})$ having law $\mathsf{M}_N^{(s)}$, is uniformly integrable.
\end{proof}

The next proposition is the key ingredient for proving uniform integrability for the sequence of random variables (\ref{RandomVariables}). 

\begin{prop}\label{DiagonalConsistency}
Let $s>-\frac{1}{2}$. Then, there exists an exchangeable sequence of random variables $\{\mathsf{d}_i\}_{i=1}^{\infty}$ with each $\mathsf{d}_i$ having the following Pearson IV  probability distribution on $\mathbb{R}$:
\begin{align*}
\frac{2^{2s}\Gamma(s+1)^2}{\pi \Gamma(2s+1)}\frac{1}{\left(1+x^2\right)^{1+s}}dx,
\end{align*}
so that, for all $N\ge 1$, we have the following equality in distribution:
\begin{align}
\sum_{i=1}^{N}\mathsf{x}_i^{(N)}\overset{\textnormal{d}}{=}\sum_{i=1}^{N}\mathsf{d}_i
\end{align}
where $(\mathsf{x}_1^{(N)},\dots,\mathsf{x}_N^{(N)})$ has law $\mathsf{M}_N^{(s)}$.
\end{prop}

\begin{proof}[Proof of Proposition \ref{DiagonalConsistency}] Let $\mathbb{H}(N)$ be the space of $N\times N$ Hermitian matrices. For $s>-\frac{1}{2}$ we define the following probability measure on $\mathbb{H}(N)$
\begin{align}
\mathfrak{M}_{N}^{(s)}\left(d\mathbf{H}\right)=\frac{1}{Z_{N}^{(s)}}\det\left(\mathbf{I}+\mathbf{H}^2\right)^{-s-N}d\mathbf{H},
\end{align}
where $d\mathbf{H}$ denotes Lebesgue measure on $\mathbb{H}(N)$, more precisely:
\begin{align*}
d\mathbf{H}=\prod_{j=1}^{N}d\mathbf{H}_{jj}\prod_{1\le j< k\le N}^{}d\Re\left(\mathbf{H}_{jk}\right)d\Im\left(\mathbf{H}_{jk}\right)
\end{align*}
and the normalization constant $Z_N^{(s)}$ (see Proposition 3.1 in \cite{BorodinOlshanski}) is given by:
\begin{align*}
Z_N^{(s)}=\prod_{j=1}^{N}\frac{\pi^j\Gamma(2s+j)}{2^{2s+2j-2}\Gamma(s+j)^2} .
\end{align*}
Denote by $\mathsf{eval}_N:\mathbb{H}(N)\to \mathbb{W}_N$ the map taking a matrix $\mathbf{H}\in \mathbb{H}(N)$ to its ordered eigenvalues $\left(\mathsf{eval}_N\left(\mathbf{H}\right)_1\ge \dots \ge \mathsf{eval}_N\left(\mathbf{H}\right)_N\right)$. Then, it is a classical fact, a direct consequence of Weyl's integration formula, see for example \cite{Forrester}, that:
\begin{align*}
\left(\mathsf{eval}_N\right)_*\mathfrak{M}_N^{(s)}=\mathsf{M}_N^{(s)}, \ \forall N\ge 1.
\end{align*}
Here, for a (measurable) map $\mathfrak{T}:\mathcal{X} \to \mathcal{Y}$, between measurable spaces $\mathcal{X}$ and $\mathcal{Y}$, and probability measure $\mathfrak{m}$ on $\mathcal{X}$, the notation $\mathfrak{T}_*\mathfrak{m}$ denotes the pushforward measure of $\mathfrak{m}$ on $\mathcal{Y}$ given by: $\left(\mathfrak{T}_*\mathfrak{m}\right)(A)=\mathfrak{m}\left(\mathfrak{T}^{-1}\left(A\right)\right)$, for $A$ measurable. 

Let $N\ge 1$ be arbitrary. Let $\left(\mathsf{x}_1^{(N)},\dots,\mathsf{x}_N^{(N)}\right)\in \mathbb{W}_N$ be distributed according to $\mathsf{M}_N^{(s)}$ and $\left(\mathsf{d}_1^{(N)},\dots,\mathsf{d}_N^{(N)}\right)$ be the vector of diagonal elements of a random matrix distributed according to $\mathfrak{M}_N^{(s)}$. Then, since the trace of a matrix is equal to the sum of its eigenvalues we get that:
\begin{align*}
\sum_{i=1}^{N}\mathsf{x}_i^{(N)}\overset{\textnormal{d}}{=}\sum_{i=1}^{N}\mathsf{d}_i^{(N)}, \ \forall N\ge 1.
\end{align*}

We will now couple all the random vectors $\big\{\left(\mathsf{d}^{(N)}_1,\dots,\mathsf{d}^{(N)}_N\right)\big\}_{N\ge 1}$ together. For each $N\ge 1$ we denote by $\mathsf{\Pi}_N^{N+1}:\mathbb{H}(N+1)\to \mathbb{H}(N)$ the corners maps:
\begin{align*}
\mathsf{\Pi}_N^{N+1}\left(\{\mathbf{H}_{ij}\}_{i,j=1}^{N+1}\right)=\{\mathbf{H}_{ij}\}_{i,j=1}^{N}.
\end{align*}
It is a remarkable fact, already alluded to in the introduction, which is essentially due to Hua \cite{Hua}, see also \cite{NeretinDuke} and Proposition  3.1 in \cite{BorodinOlshanski} in particular, that the measures $\{\mathfrak{M}_{N}^{(s)} \}_{N\ge 1}$ are consistent with respect to the corners maps:
\begin{align*}
\left(\mathsf{\Pi}_N^{N+1}\right)_*\mathfrak{M}_{N+1}^{(s)}=\mathfrak{M}_{N}^{(s)}, \ \forall N \ge 1.
\end{align*}
Thus, by Kolmogorov's theorem there exists a unique probability measure $\mathfrak{M}^{(s)}$ on the space of infinite Hermitian matrices $\mathbb{H}=\underset{\leftarrow}{\lim}\mathbb{H}(N)$ such that:
\begin{align*}
\left(\mathsf{\Pi}_N^{\infty}\right)_*\mathfrak{M}^{(s)}=\mathfrak{M}_{N}^{(s)}, \ \forall N \ge 1, \ \textnormal{ where } \mathsf{\Pi}_N^{\infty}\left(\{\mathbf{H}_{ij}\}_{i,j=1}^{\infty}\right)=\{\mathbf{H}_{ij}\}_{i,j=1}^{N}.
\end{align*}
In particular, by looking at the diagonal elements of a random matrix distributed according to $\mathfrak{M}^{(s)}$, there exists a sequence of random variables $\{\mathsf{d}_i \}_{i=1}^{\infty}$ such that
\begin{align*}
\left(\mathsf{d}_1,\dots,\mathsf{d}_N\right)\overset{\textnormal{d}}{=}\left(\mathsf{d}^{(N)}_1,\dots,\mathsf{d}^{(N)}_N\right), \ \forall N\ge 1.
\end{align*}

We next show that the sequence of random variables $\{\mathsf{d}_i \}_{i=1}^{\infty}$ is actually exchangeable. First, as mentioned in the introduction, observe that $\mathbb{U}(N)$ has a natural action on $\mathbb{H}(N)$ by conjugation: for each $\mathbf{U}\in \mathbb{U}(N)$ we have $\mathsf{T}_{\mathbf{U}}:\mathbb{H}(N)\to \mathbb{H}(N)$ given by $\mathsf{T}_{\mathbf{U}}(\mathbf{H})=\mathbf{U}^*\mathbf{H}\mathbf{U}$. Moreover, observe that by invariance of the Lebesgue measure on $\mathbb{H}(N)$:
\begin{align*}
\left(\mathsf{T}_{\mathbf{U}}\right)_*\mathfrak{M}_{N}^{(s)}=\mathfrak{M}_{N}^{(s)}, \ \forall \mathbf{U} \in \mathbb{U}(N), \forall N \ge 1.
\end{align*}
Now let $N\ge 1$ be arbitrary and let $\sigma$ be any permutation in the symmetric group $\mathfrak{S}(N)$. Let $\mathbf{P}_{\sigma}$ be the corresponding permutation matrix. Note that $\mathbf{P}_{\sigma}\in \mathbb{U}(N)$ and moreover observe that:
\begin{align*}
\left(\mathsf{T}_{\mathbf{P}_{\sigma}}(\mathbf{H})_{11},\dots,\mathsf{T}_{\mathbf{P}_{\sigma}}(\mathbf{H})_{NN}\right)=\left(\mathbf{H}_{\sigma(1)\sigma(1)},\dots,\mathbf{H}_{\sigma(N)\sigma(N)}\right).
\end{align*}
Then since
\begin{align*}
\left(\mathsf{T}_{\mathbf{P}_{\sigma}}\right)_*\mathfrak{M}_{N}^{(s)}=\mathfrak{M}_{N}^{(s)}, \ \forall \sigma \in \mathfrak{S}(N),
\end{align*}
we get that:
\begin{align*}
\left(\mathsf{d}_{\sigma(1)},\dots,\mathsf{d}_{\sigma(N)}\right)\overset{\textnormal{d}}{=}\left(\mathsf{d}_1,\dots,\mathsf{d}_N\right), \ \forall \sigma \in \mathfrak{S}(N).
\end{align*}
Finally, by the exchangeability just proven, we have that:
\begin{align*}
\mathsf{d}_i\overset{\textnormal{d}}{=}\mathsf{d}_1, \ \forall i\ge 1
\end{align*}
and by construction:
\begin{align*}
\mathsf{Law}\left(\mathsf{d}_1\right)=\mathfrak{M}^{(s)}_1(dx)=\mathsf{M}_1^{(s)}(dx)=\frac{2^{2s}\Gamma(s+1)^2}{\pi \Gamma(2s+1)}\frac{1}{\left(1+x^2\right)^{1+s}}dx.
\end{align*}
\end{proof}

We first prove the convergence statement in Theorem \ref{MainResult} in the range $s>0$.

\begin{proof}[Proof of convergence in Theorem \ref{MainResult} for $s>0$]
Let $s>0$ and $0\le h<s+\frac{1}{2}$. We claim that for any real $r$ so that $h<r<s+\frac{1}{2}$ we have:
\begin{align*}
\sup_{N\ge 1}\mathbb{E}^{(s)}_N\left[\bigg|\sum_{i=1}^{N}\frac{\mathsf{x}^{(N)}_i}{N}\bigg|^{2r}\right]<\infty.
\end{align*}
This implies uniform integrability of the sequence of random variables (\ref{RandomVariables}) and from Proposition \ref{ConvergenceProposition} gives us the statement of Theorem \ref{MainResult}. Moreover, from Proposition \ref{DiagonalConsistency} it suffices to show that:
\begin{align*}
\sup_{N\ge 1}\mathbb{E}^{(s)}\left[\bigg|\sum_{i=1}^{N}\frac{\mathsf{d}_i}{N}\bigg|^{2r}\right]<\infty.
\end{align*}
Let $N\ge 1$ be arbitrary. We now use the elementary version of Jensen's inequality, applied to the sum inside the expectation, with the function $|\cdot|^{2r}$, such an application is valid as long as $r\ge \frac{1}{2}$ which we can choose so since $s>0$. We thus obtain:

\begin{align}
\mathbb{E}^{(s)}\left[\bigg|\sum_{i=1}^{N}\frac{\mathsf{d}_i}{N}\bigg|^{2r}\right]\le \mathbb{E}^{(s)}\left[\frac{\sum_{i=1}^{N}|\mathsf{d}_i|^{2r}}{N}\right]=\mathbb{E}^{(s)}\left[|\mathsf{d}_1|^{2r}\right], \ \forall N \ge 1,
\end{align}
from which the result follows.
\end{proof}

Before continuing some comments are in order. In the course of proving Theorem \ref{MainResult} for $s>0$ we only required the fact that the $\{\mathsf{d}_i\}_{i=1}^\infty$ are identically distributed and finiteness of certain moments. Also observe that the constraint $s>0$ is exactly the condition required to have $\mathbb{E}^{(s)}\left[|\mathsf{d}_1|\right]<\infty$.  Then, it is easy to show that $\big\{ N^{-1}\sum_{i=1}^{N}\mathsf{d}_i\big\}_{N=1}^\infty$ forms a backwards martingale (with respect to the exchangeable filtration) and thus by the backwards martingale convergence theorem we obtain that this sequence converges almost surely and in $L^1$ (this is the so-called law of large numbers for exchangeable random variables). More generally, standard martingale arguments also give the convergence of the moments $\mathbb{E}^{(s)}\left[N^{-2h}\big|\sum_{i=1}^{N}\mathsf{d}_i\big|^{2h}\right]$ up to the optimal threshold $h<s+\frac{1}{2}$, which provides an alternative route to the proof presented above.

On the other hand, for $-\frac{1}{2}<s\le 0$ we are outside the domain of the law of large numbers and more delicate arguments are required to establish Theorem \ref{MainResult}. The main idea behind the proof is the fact that (as we see in the sequel) the $\{\mathsf{d}_i \}_{i=1}^\infty$ are conditionally i.i.d. random variables with a mean and variance that can be connected back to the random variable $\mathsf{X}(s)$ and point process $\mathsf{C}^{(s)}$, which can then be analysed by making use of the determinantal structure. We begin with the following elementary probabilistic lemma which is of independent interest. Its significance for the problem at hand will be clear shortly.

\begin{lem}\label{LowMomentsConvLem}
Let $0<p\le 2$. Assume that we are given a sequence of random variables $\{\mathsf{Y}_i \}_{i=1}^\infty$ whose conditional distribution on some sigma algebra $\mathcal{B}$ (possibly generated by some random variable $\mathsf{B}$) is that of i.i.d. random variables with random mean $\mathfrak{m}$ and variance $\mathfrak{v}$. Moreover, assume that the following holds:
\begin{align*}
\mathbb{E}\left(\mathfrak{v}^{\frac{p}{2}}\right)<\infty.
\end{align*}
Then, we have:
\begin{align*}
\frac{1}{N}\sum_{i=1}^{N}\mathsf{Y}_i\overset{L^p}{\longrightarrow}\mathfrak{m}, \ \textnormal{ as } N\to \infty
\end{align*}
and so in particular:
\begin{align*}
\mathbb{E}\left[\bigg|\frac{\mathsf{Y}_1+\dots+\mathsf{Y}_N}{N}\bigg|^p\right]\longrightarrow\mathbb{E}\left[|\mathfrak{m}|^p\right]<\infty.
\end{align*}
\end{lem}

\begin{proof} First assume that $\{\mathsf{Z}_i \}_{i=1}^\infty$ is an i.i.d. sequence of random variables with mean $0$ and variance $\sigma^2$. Then, by applying Jensen's inequality, since $0<p \le 2$, we have:
\begin{align*}
\mathbb{E}\left[\bigg|\sum_{i=1}^{N}\mathsf{Z}_i\bigg|^p\right]\le \mathbb{E}\left[\bigg|\sum_{i=1}^{N}\mathsf{Z}_i\bigg|^2\right]^{\frac{p}{2}} =N^{\frac{p}{2}}\left(\sigma^2\right)^{\frac{p}{2}}.
\end{align*}
Thus, by conditioning on $\mathcal{B}$, we obtain that:
\begin{align*}
\mathbb{E}\left[\bigg|\sum_{i=1}^{N}\left(\mathsf{Y}_i-\mathfrak{m}\right)\bigg|^p\right]\le N^{\frac{p}{2}}\mathbb{E}\left[\mathfrak{v}^{\frac{p}{2}}\right].
\end{align*}
Dividing through by $N^p$ we see that $N^{-1}\sum_{i=1}^{N}\mathsf{Y}_i$ converges to $\mathfrak{m}$ in $L^p$ provided that $\mathbb{E}\left[\mathfrak{v}^{\frac{p}{2}}\right]<\infty$.
\end{proof}

We now prove the convergence statement in Theorem \ref{MainResult} in the range $-\frac{1}{2}<s\le 0$ via an application of Lemma \ref{LowMomentsConvLem}.

\begin{proof}[Proof of convergence in Theorem \ref{MainResult} for $-\frac{1}{2}<s\le 0$]
First, recall from Section \ref{SubsectionOrigins}, see \cite{BorodinOlshanski}, \cite{Qiu}, \cite{OrbitalBetaBoundary} for proofs and more details, that we have the following disintegration of the Hua-Pickrell probability measure $\mathfrak{M}^{(s)}$ on $\mathbb{H}$:
\begin{align*}
\mathfrak{M}^{(s)}=\int_{\Omega}^{}\nu^{(s)}\left(d\omega\right)\mathfrak{N}_{\omega},
\end{align*}
where the probability measure $\nu^{(s)}$ is uniquely determined through this decomposition and $\mathfrak{N}_{\omega}$ is the ergodic $\mathbb{U}(\infty)$-invariant measure on $\mathbb{H}$ corresponding to $\omega \in \Omega$. Moreover, under the ergodic measure $\mathfrak{N}_{\omega}$, where $\omega=\left(\{\alpha_i^+ \}_{i=1}^\infty, \{\alpha_i^- \}_{i=1}^\infty, \gamma_1,\gamma_2\right)$, the sequence of diagonal elements $\{\mathsf{d}_i \}_{i=1}^\infty$ is i.i.d. with an explicit distribution depending on $\omega$ (that we shall not need here) having mean $\gamma_1$ and variance $\sum_{i=1}^{\infty}\left(\alpha_i^+\right)^2+\left(\alpha_i^-\right)^2+\gamma_2$, see \cite{OlshanskiVershik}, \cite{BorodinOlshanski}, \cite{OrbitalBetaBoundary} for details. 

In particular, we can sample the diagonal elements $\{\mathsf{d}_i \}_{i=1}^\infty$ under $\mathfrak{M}^{(s)}$ as follows: we first sample a random variable $\omega\in\Omega$ according to the probability measure $\nu^{(s)}$ and then conditioned on $\omega$ we sample an i.i.d. sequence $\{ \mathsf{d}_i \}_{i=1}^\infty$ with a certain explicit distribution, see \cite{OlshanskiVershik}, \cite{BorodinOlshanski}, \cite{OrbitalBetaBoundary}, with  mean $\gamma_1$ and variance $\sum_{i=1}^{\infty}\left(\alpha_i^+\right)^2+\left(\alpha_i^-\right)^2+\gamma_2$. Thus, by conditioning on $\omega$ and recalling that under the Hua-Pickrell measures $\gamma_1\overset{\textnormal{d}}{=}\mathsf{X}(s)$ (see Theorem 2.3 in \cite{Qiu}), we can use Lemma \ref{LowMomentsConvLem} to obtain that, for $-\frac{1}{2}<s\le 0$ and $0< h <s+\frac{1}{2}$:
\begin{align*}
\mathbb{E}^{(s)}_N\left[\bigg|\sum_{i=1}^{N}\frac{\mathsf{x}^{(N)}_i}{N}\bigg|^{2h}\right]=\mathbb{E}^{(s)}\left[\bigg|\sum_{i=1}^{N}\frac{\mathsf{d}_i}{N}\bigg|^{2h}\right]\overset{N \to \infty}{\longrightarrow}\mathbb{E}\left[|\mathsf{X}(s)|^{2h}\right],
\end{align*}
provided that we have:
\begin{align}\label{VarianceFiniteness}
\mathbb{E}^{(s)}\left[\left(\sum_{i=1}^{\infty}\left(\alpha_i^+\right)^2+\left(\alpha_i^-\right)^2+\gamma_2\right)^{h}\right]<\infty.
\end{align}
This then concludes the proof of the theorem in the range $-\frac{1}{2}<s\le 0$ subject to proving (\ref{VarianceFiniteness}). Recalling that from Theorem 2.1 in \cite{Qiu}, under the Hua-Pickrell measures $\gamma_2\equiv 0$ a.s. and moreover from \cite{BorodinOlshanski}, $\{\alpha_i^+ \} \sqcup \{-\alpha_i^-\} \overset{\textnormal{d}}{=} \mathsf{C}^{(s)}$, we are then required to prove that:
\begin{align*}
\mathbb{E}^{(s)}\left[\left(\sum_{i=1}^{\infty}\left(\alpha_i^+\right)^2+\left(\alpha_i^-\right)^2+\gamma_2\right)^h\right]=\mathbb{E}^{(s)}\left[\left(\sum_{i=1}^{\infty}\left(\alpha_i^+\right)^2+\left(\alpha_i^-\right)^2\right)^h\right]=\mathbb{E}\left[\left(\sum_{x\in \mathsf{C}^{(s)}}^{}x^2\right)^h\right]<\infty.
\end{align*}
We will now make use of the following elementary inequality: for $0<\eta<1$ and non-negative real numbers $\{t_i\}_{i=1}^\infty$, with the convention that $0^{\eta}=0$, we have:
\begin{align}\label{inequality}
\left(\sum_{i=1}^{\infty}t_i\right)^\eta\le \sum_{i=1}^{\infty}t_i^{\eta}.
\end{align}
Thus, we obtain:
\begin{align*}
\mathbb{E}\left[\left(\sum_{x\in \mathsf{C}^{(s)}}^{}x^2\right)^h\right]&=\mathbb{E}\left[\left(\sum_{x\in \mathsf{C}^{(s)}}^{}x^2\mathbf{1}(|x|\ge 1)+\sum_{x\in \mathsf{C}^{(s)}}^{}x^2\mathbf{1}(|x|<1)\right)^h\right]\\
&\le \mathbb{E}\left[\left(\sum_{x\in \mathsf{C}^{(s)}}^{}x^2\mathbf{1}(|x|\ge 1)\right)^h\right] +\mathbb{E}\left[\left(\sum_{x\in \mathsf{C}^{(s)}}^{}x^2\mathbf{1}(|x|< 1)\right)^h\right] \\
&\le \mathbb{E}\left[\left(\sum_{x\in \mathsf{C}^{(s)}}^{}x^2\mathbf{1}(|x|\ge 1)\right)^h\right] +\mathbb{E}\left[\sum_{x\in \mathsf{C}^{(s)}}^{}x^2\mathbf{1}(|x|< 1)\right]^{h}\\
&\le \mathbb{E}\left[\sum_{x\in \mathsf{C}^{(s)}}^{}x^{2h}\mathbf{1}(|x|\ge 1)\right] +\mathbb{E}\left[\sum_{x\in \mathsf{C}^{(s)}}^{}x^2\mathbf{1}(|x|< 1)\right]^{h},
\end{align*}
where in the first and third inequality we have used $(\ref{inequality})$ while in the second one we have used Jensen's inequality. Hence, it will suffice to show that:
\begin{align*}
\mathbb{E}\left[\sum_{x\in \mathsf{C}^{(s)}}^{}x^{2h}\mathbf{1}(|x|\ge 1)\right]<\infty, \ \mathbb{E}\left[\sum_{x\in \mathsf{C}^{(s)}}^{}x^2\mathbf{1}(|x|< 1)\right]<\infty.
\end{align*}
By definition\footnote{To be precise we apply the determinantal point process definition with the bounded Borel functions of compact support in $\mathbb{R}^*$, $G_R(x)=\mathbf{1}\left(1\le|x|<R\right)x^{2h}$ and $F_R(x)=\mathbf{1}\left(R^{-1}<|x|<1\right)x^2$ to obtain:
\begin{align*}
  \mathbb{E}\left[\sum_{x\in \mathsf{C}^{(s)}}^{}G_R(x)\right]=\int_{-\infty}^{\infty}G_R(x)\mathsf{K}^{(s)}\left(x,x\right)dx, \ \  \mathbb{E}\left[\sum_{x\in \mathsf{C}^{(s)}}^{}F_R(x)\right]=\int_{-\infty}^{\infty}F_R(x)\mathsf{K}^{(s)}\left(x,x\right)dx
\end{align*}
and then apply the monotone convergence theorem to take $R \to \infty$.} of the determinantal property of the point process $\mathsf{C}^{(s)}$ this is equivalent to:
\begin{align*}
\mathbb{E}\left[\sum_{x\in \mathsf{C}^{(s)}}^{}x^{2h}\mathbf{1}(|x|\ge 1)\right]&=\int_{|x|\ge 1}^{}x^{2h}\mathsf{K}^{(s)}\left(x,x\right)dx<\infty, \\ \mathbb{E}\left[\sum_{x\in \mathsf{C}^{(s)}}^{}x^2\mathbf{1}(|x|< 1)\right]&=\int_{|x|<1}^{}x^{2}\mathsf{K}^{(s)}\left(x,x\right)dx<\infty,
\end{align*}
which is the content of Lemma \ref{KernelEstimate} below.
\end{proof}

\begin{lem}\label{KernelEstimate}
Let $s>-\frac{1}{2}$, $h<s+\frac{1}{2}$ and $0<R<\infty$. Then, we have:
\begin{align*}
\int_{|x|\ge R}^{}x^{2h}\mathsf{K}^{(s)}\left(x,x\right)dx&<\infty,\\
\int_{|x|<R}^{}x^2\mathsf{K}^{(s)}\left(x,x\right)dx&<\infty.
\end{align*}
\end{lem}

\begin{proof}
By symmetry about zero it suffices to show that:
\begin{align*}
\int_{R}^{\infty}x^{2h}\mathsf{K}^{(s)}\left(x,x\right)dx<\infty, \ \int_{0}^{R}x^{2}\mathsf{K}^{(s)}\left(x,x\right)dx<\infty.
\end{align*}
Now, from Definition \ref{DefinitionGamma1} for $x>0$ we have the following explicit expression for $\mathsf{K}^{(s)}(x,x)$, where we have used L'Hopital's rule to resolve the singularity in the denominator:
\begin{align*}
\mathsf{K}^{(s)}(x,x)=const^{(s)}\left(x^{-\frac{1}{2}}J_{s+\frac{1}{2}}\left(\frac{1}{x}\right)\frac{d}{dx}\left[x^{-\frac{1}{2}}J_{s-\frac{1}{2}}\left(\frac{1}{x}\right)\right]-x^{-\frac{1}{2}}J_{s-\frac{1}{2}}\left(\frac{1}{x}\right)\frac{d}{dx}\left[x^{-\frac{1}{2}}J_{s+\frac{1}{2}}\left(\frac{1}{x}\right)\right]\right)
\end{align*}
and where the explicit constant $const^{(s)}$ is given by:
\begin{align*}
const^{(s)}=\frac{2^{4s-1}\Gamma(s+1)^2\Gamma\left(s+\frac{1}{2}\right)\Gamma\left(s+\frac{3}{2}\right)}{\pi\Gamma(2s+1)\Gamma(2s+2)}.
\end{align*}
Using the following formulae for the derivative of the Bessel function:
\begin{align*}
\frac{d}{dx}J_{\beta}\left(\frac{1}{x}\right)&=-\frac{1}{x^2}\left[J_{\beta-1}\left(\frac{1}{x}\right)-\beta x J_{\beta}\left(\frac{1}{x}\right)\right]\\
&=-\frac{1}{x^2}\left[\beta xJ_\beta\left(\frac{1}{x}\right)-J_{\beta+1}\left(\frac{1}{x}\right)\right]
\end{align*}
and some manipulations we obtain that:
\begin{align*}
\mathsf{K}^{(s)}(x,x)=const^{(s)}\left[x^{-3}J^2_{s+\frac{1}{2}}\left(\frac{1}{x}\right)+x^{-3}J^2_{s-\frac{1}{2}}\left(\frac{1}{x}\right)-2sx^{-2}J_{s+\frac{1}{2}}\left(\frac{1}{x}\right)J_{s-\frac{1}{2}}\left(\frac{1}{x}\right)\right].
\end{align*}
By making the change of variables $x\mapsto 1/z$ we get:
\begin{align*}
\int_{R}^{\infty}x^{2h}\mathsf{K}^{(s)}\left(x,x\right)dx&=const^{(s)}\int_{0}^{\frac{1}{R}}z^{1-2h}J^2_{s+\frac{1}{2}}\left(z\right)+z^{1-2h}J^2_{s-\frac{1}{2}}\left(z\right)-2sz^{-2h}J_{s+\frac{1}{2}}\left(z\right)J_{s-\frac{1}{2}}\left(z\right)dz,\\
\int_{0}^{R}x^2\mathsf{K}^{(s)}\left(x,x\right)dx&=const^{(s)}\int_{\frac{1}{R}}^{\infty}z^{-1}J^2_{s+\frac{1}{2}}\left(z\right)+z^{-1}J^2_{s-\frac{1}{2}}\left(z\right)-2sz^{-2}J_{s+\frac{1}{2}}\left(z\right)J_{s-\frac{1}{2}}\left(z\right)dz.
\end{align*}
Now, recall the classical asymptotics of the Bessel function for small and large argument:
\begin{align*}
J_\beta(x)&\sim \frac{x^\beta}{2^\beta\Gamma(\beta+1)}, \ \textnormal{ as } x\to 0,\\
J_\beta(x)&\sim \sqrt{\frac{2}{\pi x}}\cos\left(x-\frac{2\beta+1}{4}\pi\right)+\mathcal{O}\left(\frac{1}{x}\right), \ \textnormal{ as } x\to \infty.
\end{align*}
From these asymptotics it is immediate that we have:
\begin{align*}
z^{1-2h}J^2_{s+\frac{1}{2}}\left(z\right)+z^{1-2h}J^2_{s-\frac{1}{2}}\left(z\right)-2sz^{-2h}J_{s+\frac{1}{2}}\left(z\right)J_{s-\frac{1}{2}}\left(z\right)&=\mathcal{O}\left(z^{2s-2h}\right), \ \textnormal{ as } z\to 0, \\
z^{-1}J^2_{s+\frac{1}{2}}\left(z\right)+z^{-1}J^2_{s-\frac{1}{2}}\left(z\right)-2sz^{-2}J_{s+\frac{1}{2}}\left(z\right)J_{s-\frac{1}{2}}\left(z\right)&=\mathcal{O}\left(z^{-2}\right), \ \textnormal{ as } z\to \infty,
\end{align*}
which since $h<s+\frac{1}{2}$, gives the required integrability for both integrals and completes the proof of the lemma. 
\end{proof}

We now finally complete the proof of our main result. 

\begin{proof}[Completion of proof of Theorem \ref{MainResult}]
It remains to show that $0<\mathfrak{F}(s,h)<\infty$. The fact that $\mathfrak{F}(s,h)<\infty$ is an immediate consequence of the proofs of convergence above, while $\mathfrak{F}\left(s,h\right)>0$ follows from Lemma \ref{NotZeroLem} below.
\end{proof}

The statement of Lemma \ref{NotZeroLem} is rather intuitive, however as we shall see below its proof is quite non-trivial. 

\begin{lem}\label{NotZeroLem}
Let $s>-\frac{1}{2}$. Then, $\mathsf{X}(s)$ is not almost surely zero.
\end{lem}
\begin{proof}
We begin by observing that by construction, see for example \cite{BorodinOlshanski}, since $\mathsf{C}^{(s)}$ has a maximal in absolute value point, there exists some $R>0$ so that all the points of $\mathsf{C}^{(s)}$ are contained in $\mathbb{R}^*\cap (-R,R)$ with positive probability\footnote{In fact, the probability (\ref{ProbabilityInterval}) is positive for all $R>0$ and precise asymptotics as $R\to 0$ are known, see for example \cite{DeiftKrasovskyVasilevska}.}:
\begin{align}\label{ProbabilityInterval}
\mathbb{P}\left(|x|<R, \ \forall x \in \mathsf{C}^{(s)} \right)>0.
\end{align}
From now on we fix such an $R>0$. Then, in order to establish the statement of the lemma it suffices to prove that\footnote{Clearly for $s>\frac{1}{2}$ the statement of the lemma follows from the explicit expression for the second moment of $\mathsf{X}(s)$ in (\ref{SecondMoment}). However, for $-\frac{1}{2}<s\le \frac{1}{2}$ one needs to introduce the cut-off at $R$ and the argument becomes more complicated and less explicit.}:
\begin{align}\label{PositiveMoment}
\mathbb{E}\left[\mathsf{X}(s)^2\mathbf{1}\left(|\alpha^{\pm}_i|<R, \textnormal{ for all } i \right)\right]>0.
\end{align}
Now, from the proof of Theorem 2.3 in \cite{Qiu} (see in particular the unnumbered display that follows display (35) involving a quantity $D_{\epsilon}$) we have:
\begin{align*}
\mathbb{E}\left[\mathsf{X}(s)^2\mathbf{1}\left(|\alpha^{\pm}_i|<R, \textnormal{ for all } i \right)\right]&=\mathbb{E}\left[\lim_{N\to \infty}\left(\sum_{x\in \mathsf{C}^{(s)}}^{}x\mathbf{1}\left(N^{-2}<|x|<R\right)\right)^2\mathbf{1}\left(|x|<R, \ \forall x \in \mathsf{C}^{(s)} \right)\right]\\&=\lim_{N\to \infty}\mathbb{E}\left[\left(\sum_{x\in \mathsf{C}^{(s)}}^{}x\mathbf{1}\left(N^{-2}<|x|<R\right)\right)^2\mathbf{1}\left(|x|<R, \ \forall x \in \mathsf{C}^{(s)} \right)\right].
\end{align*}

To analyse this quantity further we make use of the important observation that, under certain assumptions, the induced measure of a determinantal measure onto the subset of configurations all of whose points lie in a subset of the original state space is again a determinantal measure, see \cite{BufetovMultiplicative}, \cite{BufetovQuasiSymmetries}. In order to make this observation precise we need some notation. We denote by $\mathsf{Mult}_{\phi}:f\to \phi f$ the operator of multiplication by $\phi$ and we let $g_R(x)=\mathbf{1}\left(|x|<R\right)$. If we write $\mathbb{P}_K$ for the determinantal measure associated to the operator $K$ (assuming it exists) then from Corollary 1 in \cite{BufetovMultiplicative} (see also Proposition 2.7 in \cite{BufetovQuasiSymmetries}), whose conditions we will check shortly, we obtain that there exists an operator $\mathfrak{K}_R^{(s)}$ giving rise to a determinantal measure $\mathbb{P}_{\mathfrak{K}^{(s)}_R}$ so that the following holds:
\begin{align}\label{RestrictedDeterminantal}
\frac{\prod_{x\in \mathsf{C}^{(s)}}g_R(x)\mathbb{P}_{\mathsf{K}^{(s)}}}{\mathbb{P}\left(|x|<R, \ \forall x \in \mathsf{C}^{(s)} \right)}=\mathbb{P}_{\mathfrak{K}^{(s)}_R}.
\end{align}
We denote by $\mathfrak{C}_R^{(s)}$ the determinantal point process associated to $\mathfrak{K}_R^{(s)}$. Now, in order for the above application of Corollary 1 in \cite{BufetovMultiplicative} to be valid we need to check two conditions. First, the operator $\mathsf{Mult}_{(g_R-\mathbf{1})}\mathsf{K}^{(s)}$ needs to be trace class which is a consequence of Lemma \ref{KernelEstimate}. Second, the operator $\mathbf{1}+\mathsf{Mult}_{(g_R-\mathbf{1})}\mathsf{K}^{(s)}$ needs to be invertible which is equivalent to showing that its Fredholm determinant $\det\left(\mathbf{1}+\mathsf{Mult}_{(g_R-\mathbf{1})}\mathsf{K}^{(s)}\right)$ is non-zero. This holds since this determinant is actually equal to the probability (\ref{ProbabilityInterval}) which is strictly positive.

In fact, the operator $\mathfrak{K}_R^{(s)}$ has the following expression in terms of $\mathsf{K}^{(s)}$, see \cite{BufetovMultiplicative}, \cite{BufetovQuasiSymmetries} (we write $\sqrt{g_R}$ rather than $g_R$, to be consistent with the notation of \cite{BufetovMultiplicative}, \cite{BufetovQuasiSymmetries} which are working in a broader setting where $g_R$ is replaced by more general functions):
\begin{align*}
\mathfrak{K}_R^{(s)}=\mathsf{Mult}_{\sqrt{g_R}}\mathsf{K}^{(s)}\left(\mathbf{1}+\mathsf{Mult}_{(g_R-\mathbf{1})}\mathsf{K}^{(s)}\right)^{-1}\mathsf{Mult}_{\sqrt{g_R}}.
\end{align*}
Although we will not make direct use of the explicit expression above, we will need a number of properties of $\mathfrak{K}^{(s)}_R$ inherited from $\mathsf{K}^{(s)}$, see \cite{BufetovMultiplicative}, \cite{BufetovQuasiSymmetries} for more details. Firstly, it is clear that $\mathfrak{K}^{(s)}_R$ is a locally trace class positive self-adjoint contraction. Furthermore, since in fact $\mathsf{K}^{(s)}$ is an operator of orthogonal projection onto a certain subspace $\mathsf{L}^{(s)}$ of $L^{2}\left(\mathbb{R},\textnormal{Leb}\right)$, see Proposition 1.5 and Theorem 1.6 of \cite{Qiu} for the details, then $\mathfrak{K}_R^{(s)}$ is also an orthogonal projection onto the closed subspace $\mathfrak{L}_R^{(s)}=\sqrt{g_R}\mathsf{L}^{(s)}$, see Propositions 2.5 and 2.7 in \cite{BufetovQuasiSymmetries}. Moreover, it is easy to show that $\mathfrak{K}^{(s)}_R$ inherits the symmetry about zero property of $\mathsf{K}^{(s)}$ and finally we note that since $\mathfrak{K}^{(s)}_R$ is a locally trace class operator it admits a kernel which, by slightly abusing notation, we denote by $\mathfrak{K}^{(s)}_R(x,y)$. 

Now, from (\ref{RestrictedDeterminantal}) we obtain:
\begin{align*}
\frac{\mathbb{E}\left[\left(\sum_{x\in \mathsf{C}^{(s)}}^{}x\mathbf{1}\left(N^{-2}<|x|<R\right)\right)^2\mathbf{1}\left(|x|<R, \ \forall x \in \mathsf{C}^{(s)} \right)\right]}{\mathbb{P}\left(|x|<R, \ \forall x \in \mathsf{C}^{(s)} \right)}=\mathbb{E}\left[\left(\sum_{x\in \mathfrak{C}_R^{(s)}}^{}x\mathbf{1}\left(N^{-2}<|x|<R\right)\right)^2\right]
\end{align*}
and thus in order to prove (\ref{PositiveMoment}) it will suffice to show that:
\begin{align}\label{PositiveMoment2}
\lim_{N\to \infty}\mathbb{E}\left[\left(\sum_{x\in \mathfrak{C}_R^{(s)}}^{}x\mathbf{1}\left(N^{-2}<|x|<R\right)\right)^2\right]>0.
\end{align}
Moreover, by making use of the determinantal property we obtain that the last display is equal to:
\begin{align*}
&\lim_{N\to \infty}\mathbb{E}\left[\sum_{x\in \mathfrak{C}_R^{(s)}}^{}x^2\mathbf{1}\left(N^{-2}<|x|<R\right)+\sum_{x\neq y\in \mathfrak{C}_R^{(s)}}^{}xy\mathbf{1}\left(N^{-2}<|x|<R\right)\mathbf{1}\left(N^{-2}<|y|<R\right)\right]\\
&=\lim_{N\to \infty}\bigg[\int_{N^{-2}<|x|<R}^{}x^2\mathfrak{K}_R^{(s)}(x,x)dx\\& \ \ \ \ \  +\int_{N^{-2}<|x|<R}^{}\int_{N^{-2}<|y|<R}^{}xy\left(\mathfrak{K}_R^{(s)}(x,x)\mathfrak{K}_R^{(s)}(y,y)-\mathfrak{K}_R^{(s)}(x,y)\mathfrak{K}_R^{(s)}(y,x)\right)dxdy\bigg]\\
&=\int_{-R}^{R}x^2\mathfrak{K}_R^{(s)}(x,x)dx-\lim_{N\to \infty}\int_{N^{-2}<|x|<R}^{}\int_{N^{-2}<|y|<R}^{}xy\mathfrak{K}_R^{(s)}(x,y)^2dxdy.
\end{align*}
In the last line we have used the fact that $\mathfrak{K}_R^{(s)}\left(x,y\right)=\mathfrak{K}_R^{(s)}\left(y,x\right)$ and that by symmetry of $\mathfrak{C}_R^{(s)}$ about zero we have that:
\begin{align*}
\int_{N^{-2}<|x|<R}^{}x\mathfrak{K}_R^{(s)}(x,x)dx=0.
\end{align*}
Now, observe that we can dominate the integrand in the second term as follows, for all $N\ge 1$:
\begin{align*}
|x|\mathbf{1}\left(N^{-2}<|x|<R\right)|y|\mathbf{1}\left(N^{-2}<|y|<R\right)\mathfrak{K}_R^{(s)}(x,y)^2\le \left(\frac{x^2+y^2}{2}\right)\mathfrak{K}_R^{(s)}\left(x,y\right)^2\mathbf{1}\left(|y|<R\right)\mathbf{1}\left(|x|<R\right)
\end{align*}
and moreover:
\begin{align*}
&\int_{\mathbb{R}^*}^{}\int_{\mathbb{R}^*}^{}\left(\frac{x^2+y^2}{2}\right)\mathfrak{K}_R^{(s)}\left(x,y\right)^2\mathbf{1}\left(|y|<R\right)\mathbf{1}\left(|x|<R\right)dxdy\\
&\le \int_{\mathbb{R}^*}^{}x^2\mathbf{1}\left(|x|<R\right)dx\int_{\mathbb{R}^*}^{}\mathfrak{K}_R^{(s)}\left(x,y\right)^2dy=\int_{-R}^{R}x^2\mathfrak{K}_R^{(s)}(x,x)dx\le\frac{\int_{-R}^{R}x^2\mathsf{K}^{(s)}(x,x)dx}{\mathbb{P}\left(|x|<R, \ \forall x \in \mathsf{C}^{(s)} \right)} <\infty.
\end{align*}
In the single equality above we have used the fact that $\mathfrak{K}_R^{(s)}$ is a kernel of orthogonal projection onto a certain closed subspace $\mathfrak{L}_R^{(s)}=\sqrt{g_R}\mathsf{L}^{(s)}$ of $L^{2}\left(\mathbb{R},\textnormal{Leb}\right)$, see Proposition 1.5 and Theorem 1.6 in \cite{Qiu} for more details, so that in particular:
\begin{align*}
\int_{\mathbb{R}^*}^{}\mathfrak{K}_R^{(s)}\left(x,y\right)\mathfrak{K}_R^{(s)}(y,z)dy=\mathfrak{K}_R^{(s)}(x,z),
\end{align*}
along with $\mathfrak{K}_R^{(s)}\left(x,y\right)=\mathfrak{K}_R^{(s)}\left(y,x\right)$. Thus, by applying the dominated convergence theorem we obtain:
\begin{align*}
\lim_{N\to \infty}\mathbb{E}\left[\left(\sum_{x\in \mathfrak{C}_R^{(s)}}^{}x\mathbf{1}\left(N^{-2}<|x|<R\right)\right)^2\right]=\int_{-R}^{R}x^2\mathfrak{K}_R^{(s)}\left(x,x\right)dx-\int_{-R}^{R}\int_{-R}^{R}xy\mathfrak{K}_R^{(s)}(x,y)^2dxdy.
\end{align*}
 Now, by applying the Cauchy-Schwarz inequality we get:
\begin{align*}
\bigg|\int_{-R}^{R}\int_{-R}^{R} xy \mathfrak{K}_R^{(s)}(x,y)^2 dx dy\bigg|
&\le \left(\int_{-R}^{R}\int_{-R}^{R}x^2\mathfrak{K}_R^{(s)}(x,y)^2dxdy\right)^{\frac{1}{2}} \left(\int_{-R}^{R}\int_{-R}^{R}y^2\mathfrak{K}_R^{(s)}(x,y)^2dxdy\right)^{\frac{1}{2}}\\
&\le \left(\int_{-R}^{R}x^2dx \int_{\mathbb{R}^*}^{}\mathfrak{K}_R^{(s)}(x,y)^2dy\right)^{\frac{1}{2}} \left(\int_{-R}^{R}y^2dy \int_{\mathbb{R}^*}^{}\mathfrak{K}_R^{(s)}(x,y)^2dx\right)^{\frac{1}{2}}\\
&=\int_{-R}^{R}x^2\mathfrak{K}_R^{(s)}\left(x,x\right)dx.
\end{align*}
However, the first inequality above is strict since equality in Cauchy-Schwarz would hold if and only if, for some non-zero constant $c$, we have:
\begin{align*}
x^2\mathfrak{K}_R^{(s)}(x,y)^2=cy^2\mathfrak{K}_R^{(s)}(x,y)^2, \ \textnormal{ for Lebesgue almost every } (x,y)\in \left(\mathbb{R}^*\cap (-R,R)\right)^2,
\end{align*}
which is obviously false. This gives (\ref{PositiveMoment2}) and completes the proof.
\end{proof}

\section{Proof of results for the leading order coefficient}\label{SectionPainleveProof}

We first prove Proposition \ref{EvenMomentsFormulaProposition} using the ideas developed in the proof of Theorem \ref{MainResult}.

\begin{proof}[Proof of Proposition \ref{EvenMomentsFormulaProposition}] Let $h\in \mathbb{N}$. We begin with the following important observation:
\begin{align}\label{MomentsInTermsOfDiagonal}
\mathbb{E}\left[\mathsf{X}(s)^{2h}\right]=\mathbb{E}^{(s)}\left[\mathsf{d}_1\mathsf{d}_2\cdots\mathsf{d}_{2h}\right], \ s>h-\frac{1}{2}.
\end{align}
By making use of the disintegration of $\mathfrak{M}^{(s)}$ explained in the proof of Theorem \ref{MainResult} this can be seen as follows. By conditioning on $\omega\in \Omega$, we have that the $\{\mathsf{d}_i\}_{i=1}^\infty$ are i.i.d. with mean $\gamma_1$ and moreover recalling that under $\mathfrak{M}^{(s)}$ we have $\gamma_1\overset{\textnormal{d}}{=}\mathsf{X}(s)$ establishes (\ref{MomentsInTermsOfDiagonal}). Thus, by using the fact that for $s>h-\frac{1}{2}$:
\begin{align*}
\mathbb{E}_k^{(s)}\left[\left(\mathsf{x}_1^{(k)}+\dots+\mathsf{x}_k^{(k)}\right)^{2h}\right]=\mathbb{E}^{(s)}\left[\left(\mathsf{d}_1+\mathsf{d}_2+\dots+\mathsf{d}_k\right)^{2h}\right], \ \forall k\ge 1,
\end{align*}
in order to establish (\ref{EvenMomentsFormula}) it is equivalent to prove that:
\begin{align}\label{EquivalentEvenMomentsFormula}
\mathbb{E}^{(s)}\left[\mathsf{d}_1\mathsf{d}_2\cdots\mathsf{d}_{2h}\right]=\frac{1}{(2h)!}\sum_{k=1}^{2h}(-1)^{2h-k}\binom{2h}{k}\mathbb{E}^{(s)}\left[\left(\mathsf{d}_1+\mathsf{d}_2+\dots+\mathsf{d}_k\right)^{2h}\right].
\end{align}
We will now make use of an elementary combinatorial identity. Let $N\ge 1$ and $z_1,\dots,z_N$ be arbitrary complex numbers. For $k=1,\dots, N$ define the following quantities:
\begin{align*}
\mathcal{S}_k\left(z_1,\dots,z_N\right)=\sum_{1\le i_1<i_2<\dots<i_k\le N}^{}\left(z_{i_1}+z_{i_2}+\dots+z_{i_k}\right)^{N}.
\end{align*}
Then, we have:
\begin{align}\label{identity}
z_1z_2\cdots z_N=\frac{1}{N!}\sum_{k=1}^{N}(-1)^{N-k}\mathcal{S}_k\left(z_1,\dots,z_N\right).
\end{align}
This identity can be checked as follows. Observe that the right hand side is a polynomial of degree at most $N$ and moreover each $z_i$ is a factor. Then, checking the coefficient of $z_1z_2\cdots z_N$ gives the conclusion.

Finally, by the exchangeability of the random variables $\{\mathsf{d}_i\}_{i=1}^\infty$ we have that, for all $k=1,\dots,2h$:
\begin{align*}
\mathbb{E}^{(s)}\left[\mathcal{S}_k\left(\mathsf{d}_1,\dots,\mathsf{d}_{2h}\right)\right]=\binom{2h}{k}\mathbb{E}^{(s)}\left[\left(\mathsf{d}_1+\mathsf{d}_2+\dots+\mathsf{d}_k\right)^{2h}\right],
\end{align*}
which along with identity (\ref{identity}) immediately gives (\ref{EquivalentEvenMomentsFormula}).

We now prove the rationality claim of the proposition. By definition we have:
\begin{align*}
\mathbb{E}_N^{(s)}\left[\left(\mathsf{x}_1^{(N)}+\dots+\mathsf{x}_N^{(N)}\right)^{2h}\right]=\frac{1}{N!\mathsf{c}_N^{(s)}}\int_{-\infty}^{\infty}\cdots\int_{-\infty}^{\infty}\frac{\left(x_1+\dots+x_N\right)^{2h}\prod_{1\le i<j\le N}^{}(x_i-x_j)^2}{\left(1+x_1^2\right)^{N+s}\cdots \left(1+x_N^2\right)^{N+s}}dx_1\dots dx_N.
\end{align*}
By simply expanding the numerator:
\begin{align*}
\left(x_1+\dots+x_N\right)^{2h}\prod_{1\le i<j\le N}^{}(x_i-x_j)^2=\textnormal{Polynomial in the variables }\ x_1,\dots,x_N,
\end{align*}
we obtain that $\mathbb{E}_N^{(s)}\left[\left(\mathsf{x}_1^{(N)}+\dots+\mathsf{x}_N^{(N)}\right)^{2h}\right]$ can be written as a linear combination, with coefficients not depending on $s$, of terms of the form:
\begin{align*}
\frac{1}{\mathsf{c}_N^{(s)}}\int_{-\infty}^{\infty}dx_1\cdots \int_{-\infty}^{\infty}dx_N \frac{x_1^{2m_1}}{\left(1+x_1^2\right)^{N+s}}\cdots \frac{x_N^{2m_N}}{\left(1+x_N^2\right)^{N+s}}, \ \ m_i\in \{0,1,2, \dots\} \textnormal{ with } m_i<N+s-\frac{1}{2}.
\end{align*} 
Note that, we have no odd exponents in the integrands above since by symmetry about zero the integrals vanish:
\begin{align*}
\int_{-\infty}^{\infty}\frac{x^{2m+1}}{\left(1+x^2\right)^{N+s}}dx=0, \  m=0,1,2,\dots \ \textnormal{ with } m<N+s-1.
\end{align*}
Moreover, recall that we have the following well-known integral evaluation:
\begin{align*}
\int_{-\infty}^{\infty}\frac{x^{2m}}{\left(1+x^2\right)^{N+s}}dx=\frac{\Gamma\left(m+\frac{1}{2}\right)\Gamma\left(N+s-m-\frac{1}{2}\right)}{\Gamma(N+s)}, \  m=0,1,2,\dots \ \textnormal{ with } m<N+s-\frac{1}{2}
\end{align*}
and that from Definition \ref{DefinitionHuaPickrell}:
\begin{align*}
\mathsf{c}_N^{(s)}=(2\pi)^N 2^{-N^2} G(N+1)2^{-2sN}\prod_{j=1}^{N}\frac{\Gamma(2s+N-j+1)}{\Gamma(s+N-j+1)^2}.
\end{align*}
Hence, we are required to show that:
\begin{align*}
2^{2sN}\prod_{j=1}^{N}\frac{\Gamma(s+N-j+1)^2}{\Gamma(2s+N-j+1)}\frac{\Gamma\left(N+s-m_j-\frac{1}{2}\right)}{\Gamma(N+s)}
\end{align*}
is a rational function of $s$. Now, observe that for any $j=1,\dots, N$ the function 
\begin{align*}
\frac{\Gamma(s+N-j+1)}{\Gamma(s+N)} \ \textnormal{ is rational in } s.
\end{align*}
So the claim further reduces to showing that the following is rational in $s$:
\begin{align*}
2^{2sN}\prod_{j=1}^{N}\frac{\Gamma(s+N-j+1)\Gamma\left(N+s-m_j-\frac{1}{2}\right)}{\Gamma(2s+N-j+1)}.
\end{align*}
We now use the Legendre duplication formula:
\begin{align*}
\Gamma(z)\Gamma\left(z+\frac{1}{2}\right)=2^{1-2z}\sqrt{\pi}\Gamma(2z),
\end{align*}
to obtain that the expression above is equal to, with $const$ being independent of $s$:
\begin{align*}
const \times 2^{2sN}\prod_{j=1}^{N}\frac{\Gamma(s+N-j+1)\Gamma\left(s+N-m_j-\frac{1}{2}\right)}{2^{2s}\Gamma\left(s+\frac{N-j}{2}+\frac{1}{2}\right)\Gamma\left(s+\frac{N-j}{2}+1\right)}.
\end{align*}
The claim then follows if for all $j=1,\dots,N$ we show that the following expression is rational in $s$:
\begin{align*}
\frac{\Gamma(s+N-j+1)\Gamma\left(s+N-m_j-\frac{1}{2}\right)}{\Gamma\left(s+\frac{N-j}{2}+\frac{1}{2}\right)\Gamma\left(s+\frac{N-j}{2}+1\right)},
\end{align*}
which  is easily seen to be true since exactly one of $\frac{N-j}{2}+\frac{1}{2}$ and $\frac{N-j}{2}+1$ is an integer and the other a half integer.
\end{proof}

We now prove Proposition \ref{PainleveCharacteristicFunction}. This is done by taking the scaling limit as $N\to \infty$, using the methods developed in this paper, of a differential equation for finite $N$ from \cite{JointMomentsPainleve}.

\begin{proof}[Proof of Proposition \ref{PainleveCharacteristicFunction}]
We begin by considering (we shall show next that these quantities are well defined), with $s>1$ (not necessarily integer yet):
\begin{align*}
\mathsf{\Xi}_N^{(s)}(t)=t\frac{d}{dt}\log \mathbb{E}_N^{(s)}\left[e^{\i \frac{t}{2}\frac{\sum_{i=1}^N \mathsf{x}_i^{(N)}}{N}}\right].
\end{align*}
An elementary calculation shows that for a function $\phi(\cdot)$, with $\phi(t)\neq 0$, smooth enough we have:
\begin{align*}
t\frac{d}{dt}\log \phi(t)&=t\frac{\phi'(t)}{\phi(t)}, \\ \frac{d}{dt}\left[t\frac{d}{dt}\log \phi(t)\right]&=\frac{\phi'(t)}{\phi(t)}+t\frac{\phi''(t)}{\phi(t)}-\frac{t\phi'(t)^2}{\phi(t)^2},\\
\frac{d^2}{dt^2}\left[t\frac{d}{dt}\log \phi(t)\right]&=2\frac{\phi''(t)}{\phi(t)}-2\frac{\phi'(t)^2}{\phi(t)^2}-3t\frac{\phi'(t)\phi''(t)}{\phi(t)^2}+t\frac{\phi'''(t)}{\phi(t)}+2t\frac{\phi'(t)^3}{\phi(t)^3}.
\end{align*}
Now, from basic properties of characteristic functions we have that there exists some $T'>0$ such that:
\begin{align*}
t\mapsto  \mathbb{E}\left[e^{\i \frac{t}{2}\mathsf{X}(s)}\right] \neq 0, \ \forall t\in [0,T'].
\end{align*}
On the other hand, from Theorem \ref{BorodinOlshanskiQiu} we have that uniformly on compact sets in $\mathbb{R}$:
\begin{align*}
\mathbb{E}_N^{(s)}\left[e^{\i \frac{t}{2}\frac{\sum_{i=1}^N \mathsf{x}_i^{(N)}}{N}}\right]\overset{N\to\infty}{\longrightarrow}\mathbb{E}\left[e^{\i \frac{t}{2}\mathsf{X}(s)}\right]
\end{align*}
and thus, since for all $N\ge 1$ each of $t\mapsto \mathbb{E}_N^{(s)}\left[e^{\i \frac{t}{2}\frac{\sum_{i=1}^N \mathsf{x}_i^{(N)}}{N}}\right]$ is a characteristic function, there exists some $T\le T'$ such that:
\begin{align*}
t \mapsto \mathbb{E}_N^{(s)}\left[e^{\i \frac{t}{2}\frac{\sum_{i=1}^N \mathsf{x}_i^{(N)}}{N}}\right] \neq 0, \ \forall t\in [0,T], \ \forall{ N \ge 1}.
\end{align*}
Moreover, since $s>1$, we have $\mathbb{E}_N^{(s)}\left[|\sum_{i=1}^N \mathsf{x}_i^{(N)}|^3\right]<\infty$ and so from the formulae above we obtain that for all $N\ge 1$, the functions $t\mapsto \frac{d^p}{dt^p}\mathsf{\Xi}_N^{(s)}(t)$ for $p\in \{0,1,2\}$ are well-defined and continuous in $[0,T]$. 

Then, making use of Theorem \ref{BorodinOlshanskiQiu} and the uniform integrability of the sequence of random variables $\bigg\{ \bigg|\sum_{i=1}^{N}\frac{\mathsf{x}^{(N)}_i}{N}\bigg|^{3} \bigg\}_{N\ge 1}$ for $s>1$, with $\left(\mathsf{x}_1^{(N)},\dots,\mathsf{x}_N^{(N)}\right)$ following $\mathsf{M}_N^{(s)}$, proven in the previous section we obtain that uniformly in $t\in [0,T]$:
\begin{align*}
\frac{d^p}{dt^p}\mathbb{E}_N^{(s)}\left[e^{\i \frac{t}{2}\frac{\sum_{i=1}^N \mathsf{x}_i^{(N)}}{N}}\right]\overset{N\to\infty}{\longrightarrow}\frac{d^p}{dt^p}\mathbb{E}\left[e^{\i \frac{t}{2}\mathsf{X}(s)}\right], \ p\in \{ 0,1,2, 3\}.
\end{align*} 
Thus, from the workings above we get that, for $s>1$, uniformly in $t\in [0,T]$:
\begin{align}\label{ConvergenceCharacteristic}
\frac{d^p}{dt^p}\mathsf{\Xi}_N^{(s)}(t)\overset{N\to\infty}{\longrightarrow}\frac{d^p}{dt^p}\mathsf{\Xi}^{(s)}(t), \ p\in \{ 0,1,2\}.
\end{align}
In particular, $\mathsf{\Xi}^{(s)}$ is $C^2$ in $[0,T]$.

We now make use of the following result which is essentially implicit\footnote{By combining Proposition 2 and Theorem 1 in \cite{JointMomentsPainleve} along with display (4.22) therein (note that there is a typo in (4.22), $(-1)^{\frac{s(s-1)}{2}+h}$ should be $(-1)^{\frac{s(s-1)}{2}}$) and some straightforward elementary manipulations we obtain the following representation:
\begin{align*}
\mathbb{E}_N^{(s)}\left[e^{\i \frac{t}{2}\sum_{i=1}^N \mathsf{x}_i^{(N)}}\right]=\frac{1}{(2\pi \i)^s}\frac{(-1)^{\frac{s(s-1)}{2}}}{\mathfrak{F}_N(s,0)}e^{-\left(\frac{N}{2}+s\right)t}\times (2\pi \i)^se^{\left(\frac{N}{2}+s\right)t}(-1)^{\frac{s(s-1)}{2}}G_N^{(s)}(t)=\frac{1}{\mathfrak{F}_N(s,0)}G_N^{(s)}(t),
\end{align*}
for a certain function $G_N^{(s)}$; this is exactly the function $f_k$ from (4.22) in \cite{JointMomentsPainleve} with the identification $k=s$ (the dependence on $N$ is dropped there). Moreover, from display (4.25) therein we can write (since $G_N^{(s)}(0)=\mathfrak{F}_N(s,0)$ from display (4.28) in \cite{JointMomentsPainleve}):
\begin{align*}
\mathbb{E}_N^{(s)}\left[e^{\i \frac{t}{2}\sum_{i=1}^N \mathsf{x}_i^{(N)}}\right]=\frac{1}{\mathfrak{F}_N(s,0)}G_N^{(s)}(0)\exp \left(\int_{0}^{t}\frac{\xi_N^{(s)}(x)}{x}dx\right)=\exp \left(\int_{0}^{t}\frac{\xi_N^{(s)}(x)}{x}dx\right),
\end{align*}
where the function $\xi_N^{(s)}$ after the rescaling $\mathsf{\Xi}_N^{(s)}(x)=\xi_N^{(s)}\left(\frac{x}{N}\right)$ satisfies (\ref{PainleveFiniteN}), see displays (4.31) and (4.32) in \cite{JointMomentsPainleve}.
} in \cite{JointMomentsPainleve}. Namely, we have the following representation, for $s\in \mathbb{N}$:
\begin{align*}
\mathbb{E}_N^{(s)}\left[e^{\i \frac{t}{2}\sum_{i=1}^N \mathsf{x}_i^{(N)}}\right]=\exp \left(\int_{0}^{t}\frac{\xi_N^{(s)}(x)}{x}dx\right),
\end{align*}
for a certain function $\xi_N^{(s)}$ which satisfies a particular differential equation that we will now state, after performing a simple rescaling. Observe that, by a change of variables, we have $\mathsf{\Xi}_N^{(s)}(x)=\xi_N^{(s)}\left(\frac{x}{N}\right)$ and from Section 4.3 in \cite{JointMomentsPainleve}, see in particular display (4.32) therein, this satisfies the equation:
\begin{align}
\left(t\frac{d^2\mathsf{\Xi}_N^{(s)}}{dt^2}\right)^2=&-4t\left(\frac{d\mathsf{\Xi}_N^{(s)}}{dt}\right)^3+(4s^2+4\mathsf{\Xi}^{(s)}+N^{-2}t^2)\left(\frac{d\mathsf{\Xi}_N^{(s)}}{dt}\right)^2+t\left(1+2N^{-1}s-2N^{-2}\mathsf{\Xi}_N^{(s)}\right)\frac{d\mathsf{\Xi}_N^{(s)}}{dt}\nonumber\\
&-\left(1+2N^{-1}s-N^{-2}\mathsf{\Xi}_N^{(s)}\right)\mathsf{\Xi}_N^{(s)}.\label{PainleveFiniteN}
\end{align}
By taking the limit $N\to \infty$ in the equation (\ref{PainleveFiniteN}) above and using the uniform convergence in $[0,T]$ from (\ref{ConvergenceCharacteristic}) we get that for $s\in \mathbb{N}_{>1}$:
\begin{align*}
\left(t\frac{d^2\mathsf{\Xi}^{(s)}}{dt^2}\right)^2=-4t\left(\frac{d\mathsf{\Xi}^{(s)}}{dt}\right)^3+(4s^2+4\mathsf{\Xi}^{(s)})\left(\frac{d\mathsf{\Xi}^{(s)}}{dt}\right)^2+t\frac{d\mathsf{\Xi}^{(s)}}{dt}-\mathsf{\Xi}^{(s)}.
\end{align*}
Finally, to check the boundary conditions for $\mathsf{\Xi}^{(s)}$ observe that we have the following boundary conditions for $\mathsf{\Xi}_N^{(s)}$, for all $N\ge 1$:
\begin{align*}
\mathsf{\Xi}_N^{(s)}(0)&=t\frac{\mathbb{E}_N^{(s)}\left[\frac{\i\sum_{i=1}^N \mathsf{x}_i^{(N)}}{2N}e^{\i \frac{t}{2}\frac{\sum_{i=1}^N \mathsf{x}_i^{(N)}}{N}}\right]}{\mathbb{E}_N^{(s)}\left[e^{\i \frac{t}{2}\frac{\sum_{i=1}^N \mathsf{x}_i^{(N)}}{N}}\right]}\bigg|_{t=0}=0\\
\frac{d}{dt}\mathsf{\Xi}_N^{(s)}(t)\big|_{t=0}&=\frac{\mathbb{E}_N^{(s)}\left[\frac{\i\sum_{i=1}^N \mathsf{x}_i^{(N)}}{2N}e^{\i \frac{t}{2}\frac{\sum_{i=1}^N \mathsf{x}_i^{(N)}}{N}}\right]}{\mathbb{E}_N^{(s)}\left[e^{\i \frac{t}{2}\frac{\sum_{i=1}^N \mathsf{x}_i^{(N)}}{N}}\right]}\bigg|_{t=0}+t\frac{\mathbb{E}_N^{(s)}\left[\left(\frac{\i\sum_{i=1}^N \mathsf{x}_i^{(N)}}{2N}\right)^2e^{\i \frac{t}{2}\frac{\sum_{i=1}^N \mathsf{x}_i^{(N)}}{N}}\right]}{\mathbb{E}_N^{(s)}\left[e^{\i \frac{t}{2}\frac{\sum_{i=1}^N \mathsf{x}_i^{(N)}}{N}}\right]}\bigg|_{t=0}\\& \ \ -t\frac{\mathbb{E}_N^{(s)}\left[\frac{\i\sum_{i=1}^N \mathsf{x}_i^{(N)}}{2N}e^{\i \frac{t}{2}\frac{\sum_{i=1}^N \mathsf{x}_i^{(N)}}{N}}\right]^2}{\mathbb{E}_N^{(s)}\left[e^{\i \frac{t}{2}\frac{\sum_{i=1}^N \mathsf{x}_i^{(N)}}{N}}\right]^2}\bigg|_{t=0}=0
\end{align*}
since by symmetry around the origin $\mathbb{E}_N^{(s)}\left[\sum_{i=1}^N \mathsf{x}_i^{(N)}\right]=0$ and moreover  $\mathbb{E}_N^{(s)}\left[|\sum_{i=1}^N \mathsf{x}_i^{(N)}|^2\right]<\infty$ since $s>\frac{1}{2}$. Then, the corresponding boundary conditions for $\mathsf{\Xi}^{(s)}$ follow from (\ref{ConvergenceCharacteristic}).
\end{proof}

\bigskip
\noindent
{\sc School of Mathematics, University of Edinburgh, James Clerk Maxwell Building, Peter Guthrie Tait Rd, Edinburgh EH9 3FD, U.K.}\newline
\href{mailto:theo.assiotis@ed.ac.uk}{\small theo.assiotis@ed.ac.uk}

\bigskip
\noindent
{\sc Mathematical Institute, University of Oxford, Oxford, OX2 6GG, UK.}\newline
\href{mailto:jon.keating@maths.ox.ac.uk}{\small jon.keating@maths.ox.ac.uk}

\bigskip
\noindent
{\sc Department of Statistics, University of Warwick, Coventry CV4 7AL, U.K.}\newline
\href{mailto:J.Warren@warwick.ac.uk}{\small J.Warren@warwick.ac.uk}

\end{document}